\documentclass[12 pt]{amsart}

\usepackage{amsmath,amssymb,amsthm,amsfonts,setspace,hyperref,dsfont,yhmath,euscript}

\usepackage{color}

\newcommand{\NN}{\mathbb{N}}
\newcommand{\RR}{\mathbb{R}}
\newcommand{\R}{\mathbb{R}}

\newcommand{\CC}{\mathbb{C}}

\newcommand{\TT}{\mathbb{T}}
\newcommand{\T}{\mathbb{T}}
\newcommand{\ZZ}{\mathbb{Z}}
\newcommand{\Z}{\mathbb{Z}}

\newcommand{\norm}[1]{\lVert#1\rVert}

\newtheorem{theorem}{Theorem}[section]
\newtheorem{corollary}[theorem]{Corollary}

\newtheorem{lemma}[theorem]{Lemma}
\newtheorem{proposition}[theorem]{Proposition}

\newcommand{\comment}[1]{}
\newtheorem{definition}[theorem]{Definition}

\numberwithin{equation}{section}

\def\a{\alpha}

\def\b{\beta}
\def\s{\sigma}
\def\la{\lambda}

\def\M{\mathcal{M}}
\def\w{\mathcal{W}}
\def\E{\mathcal{E}}

\def\R{\mathbb R}
\def\Q{\mathbb Q}
\def\C{\mathbb C}
\def\Z{\mathbb Z}
\def\N{\mathbb N}
\def\T{\mathbb T}

\def\dist{\text{dist}}
\def\Id{\text{Id}}
\def\e{\epsilon}

\makeatletter
\renewcommand*\env@matrix[1][*\c@MaxMatrixCols c]{%
  \hskip -\arraycolsep
  \let\@ifnextchar\new@ifnextchar
  \array{#1}}

\makeatletter
\renewcommand*\env@matrix[1][*\c@MaxMatrixCols c]{%
  \hskip -\arraycolsep
  \let\@ifnextchar\new@ifnextchar
  \array{#1}}
\makeatother
\begin{document}
\title[Global smooth rigidity for toral automorphisms]
{Global smooth rigidity for toral automorphisms}

\author[Boris Kalinin$^1$ \and Victoria Sadovskaya$^2$ \and Zhenqi Jenny Wang$^3$ ]{Boris Kalinin$^1$ \and Victoria Sadovskaya$^2$ \and Zhenqi Jenny Wang$^3$ }


\address{Department of Mathematics, The Pennsylvania State University, University Park, PA 16802, USA.}
\email{kalinin@psu.edu, sadovskaya@psu.edu}

\address{Department of Mathematics\\
        Michigan State University\\
        East Lansing, MI 48824,USA}
\email{wangzq@math.msu.edu}


\thanks{$^1$  Supported in part by Simons Foundation grant 855238}
\thanks{$^2$ Supported in part by  Simons Foundation grant MP-TSM-00002874}
\thanks{$^3$ Supported in part by NSF grant DMS-1845416}

\begin{abstract}
We study regularity of a conjugacy between a hyperbolic or partially hyperbolic toral automorphism $L$ and a $C^\infty$ diffeomorphism $f$ of the torus. For a 
very weakly irreducible hyperbolic automorphism $L$ we show that any $C^1$ conjugacy is $C^\infty$.
For a very weakly irreducible  ergodic partially hyperbolic automorphism $L$  we show that any $C^{1+\text{H\"older}}$ conjugacy is $C^\infty$. As a corollary, we improve regularity of the conjugacy to $C^\infty$ in prior local and global rigidity results.
\end{abstract}

\date{\today}

\maketitle

\section{Introduction and main results}

\subsection{Hyperbolic systems on tori and their topological classification}$\;$\\
Hyperbolic automorphisms of tori are the prime examples of hyperbolic dynamical systems.
Any matrix $L \in GL(d,\Z)$ induces an automorphism of the torus $\T^d=\R^d/\Z^d$, which 
we denote by the same letter. An automorphism $L$ is {\em hyperbolic}
or {\em Anosov} if the matrix has no eigenvalues on the unit circle. In this case, $\R^d=E^s\oplus E^u$,
where $E^{s/u}$ is the sum of generalized eigenspaces of $L$ corresponding to the eigenvalues 
of modulus less/grater than 1. In general, a diffeomorphism $f$ of a compact manifold $\M$ is {\em Anosov} if there exist 
a continuous $Df$-invariant splitting $T\M=\E^s\oplus \E^u$ and constants $K>0$ and 
$\theta<1$ such that for all $n\in \N$,
$$
\| Df^n(v) \|  \leq K\theta^n \| v \|
     \;\text{ for all }v \in \E^s , \;\text{ and }\;\, \| Df^{-n}(v) \| \leq K\theta^n  \| v \|
     \;\text{ for all }v \in \E^u. 
$$  
\noindent
The sub-bundles $\E^s$ and $\E^u$ are called stable and unstable. They are tangent to the stable and unstable foliations $\w^s$ and $\w^u$. 


Classical result of Franks and Manning \cite{F,M} establish topological classification of Anosov diffeomorphisms of $\T^d$. Any such diffeomorphism $f$ is topologically conjugate to the hyperbolic automorphism $L$ that $f$ induces on $\Z^d=H_1(\T^d,\Z)$. In particular, a hyperbolic automorphisms $L$ is topologically conjugate to any $C^1$-small perturbation.
 A {\em topological conjugacy} is
 a ho\-meo\-morphism $H$ of $\T^d$ such that 
\begin{equation} \label{Conj def}
L\circ H= H \circ f.
\end{equation}
Any two such conjugacies differ by an affine automorphism of $\T^d$ commuting with $L$ \cite{W},
and in particular have the same regularity.

\subsection{Regularity of conjugacy for hyperbolic  automorphisms.}$\;$\\ 
A conjugacy $H$ in \eqref{Conj def} is always bi-H\"older, but it is usually not even $C^1$, as there are various obstructions to smoothness.  For example, $C^1$ regularity of $H$ requires that 
Lyapunov exponents of $L$ equal Lyapunov exponents of $f$ with respect to any measure,
and in particular at periodic orbits. Thus smoothness of the conjugacy is a rare phenomenon 
and it is often referred to as rigidity. The problem of regularity of $H$ has been extensively studied
and it can be rougly split into two questions:
 
(Q1) Is the conjugacy $C^1$ assuming vanishing of some natural obstructions? 

(Q2) If the conjugacy is $C^1$, is it $C^\infty$ for a $C^\infty$ diffeomorphism $f$?

\noindent These questions were often considered in a local setting were $f$ is a small perturbation of $L$.
Answering them in the global setting would give a description 
of  Anosov diffeomorphisms smoothly conjugate to algebraic models.

In dimension two, complete positive answers to both questions were obtained by de la Llave, Marco, and Moriy\'on \cite{LMM86,L87,L92}. 

 The higher dimensional case is much more complicated and answers to both questions 
 are negative in general, without some irreducibility assumption on $L$.
Examples by de la Llave in  \cite{L92} demonstrated that a conjugacy between a hyperbolic $L$ and its analytic perturbation $f$ can be $C^k$ but is not $C^{k+1}$, and also that vanishing of periodic obstructions
 may not yield $C^1$ conjugacy. Automorphisms $L$ in these examples are products, 
and hence  reducible.
We recall that $L$ is {\em irreducible}\, if it has no nontrivial rational invariant subspace or, equivalently, if its characteristic polynomial is irreducible over $\Q$. 

For large classes of irreducible $L$, positive answers to question (Q1) were obtained by Gogolev, Guysinky, Kalinin, Sadovskaya, Saghin, Yang, and DeWitt \cite{GG,G, GKS11, SY, GKS20,DW,DW2} under various assumptions on Lyapunov exponents and periodic data.
These results were for the local setting, and  recently they were extended to the global setting
by DeWitt and Gogolev in \cite{DWG}. 
In contrast to the two-dimensional case, they established only $C^{1+\text{H\"older}}$ regularity of $H$. 
Smoothness of $H$ on $\T^2$ was obtained along one-dimensional stable and 
unstable foliations, which have $C^\infty$ leaves. 
The higher-dimensional results for $L$ with more than one (un)stable Lyapunov exponent showed regularity
of $H$ along intermediate invariant foliations corresponding to the Lyapunov splitting of $L$.
However, the leaves of these foliations are typically only $C^{1+\text{H\"older}}$ and hence 
studying regularity along them cannot yield higher smoothness. 

Nevertheless, Gogolev conjectured \cite{G,DWG} that the answer to question (Q2) is positive for an irreducible $L$. 
Until recently, the only progress in this direction for non-conformal case was the result of Gogolev 
for automorphisms of $\T^3$ with real spectrum \cite{G17}. The key step of his proof was showing that the leaves of the one-dimensional intermediate foliation are $C^{\infty}$ by studying certain cohomological 
 equation over a Diophantine translation on $\T^3$. 

Recently in \cite{KSW}  we used KAM-type techniques to obtained a positive answer to question (Q2) for perturbations {\it weakly irreducible} hyperbolic $L$. We call a matrix $L\in GL(d,\Z)$ weakly irreducible if all factors over $\Q$ of its characteristic polynomial have the same set of moduli of their roots.  
This assumption is strictly weaker than irreducibility and it is satisfied by  some products. 
The KAM-type  techniques required that the diffeomorphism $f$ is close to $L$ in 
$C^r$ topology, where $r$  depends on $L$ and is typically quite large. Thus this result did not
give a conclusive answer to question (Q2) even for the standard local setting with $C^1$ closeness.

 In our new results, we use a completely different approach 
which allows us to avoid any closeness assumption, and so we obtain a general global regularity result
under a very weak irreducibility assumption.
For a matrix $L\in GL(d,\Z)$ we denote the largest modulus of its eigenvalues by $\rho_{\max}$
and the smallest by $\rho_{\min}$. We say that $L$ is {\em very weakly irreducible} if every 
 factor over $\Q$ of the characteristic polynomial of $L$ has a root of modulus $\rho_{\max}$
 and  a root of modulus $\rho_{\min}$. 
 Clearly this condition is weaker than  weak irreducibility. 
 Essentially, for a toral automorphism $L$ it means that any algebraic factor of $L$ has the same largest and smallest Lyapunov exponents as $L$. 
 
 Now we state our main result in the hyperbolic case.

\begin{theorem} \label{mainA}
Let $L$ be a very weakly irreducible hyperbolic automorphism of $\T^d$ and let $f$ be a $C^\infty$  diffeomorphism of $\T^d$.
If  $f$ and $L$ are conjugate by a $C^1$ diffeomorphism $H$ then any conjugacy between $f$ and $L$ is a $C^\infty$ diffeomorphism. 
\end{theorem}

Of course $f$ in the theorem is Anosov, but we do not assume that it is close to $L$.
For a $C^1$-small perturbation of $L$, certain weak differentiability of $H$ implies that $H$ is $C^{1+\text{H\"older}}$  \cite[Theorem 1.1]{KSW}, and so we obtain the following corollary  for the local setting.

\begin{corollary} \label{corollaryA}
Let $L$ be a very weakly irreducible hyperbolic automorphism of $\T^d$ and let $f$ be a $C^\infty$  diffeomorphism of $\T^d$  which is $C^1$-close to $L$. If for some conjugacy $H$ between $f$ and $L$ either $H$ or $H^{-1}$ is Lipschitz, or more generally is in a Sobolev space $W^{1,q}(\T^d)$ with $q>d$, then any conjugacy between $f$ and $L$ is a $C^\infty$ diffeomorphism. 
\end{corollary}

\subsection{Regularity of conjugacy for ergodic partially hyperbolic automorphisms.} $\;$\\
 A toral automorphism $L$ is called {\em partially hyperbolic} if the matrix has some (but not all)  eigenvalues on the unit circle. It is ergodic  if and only if none of its eigenvalues is a root of unity.
In contrast to the hyperbolic case, there is no topological classification of partially hyperbolic toral diffeomorphisms. The presence of neutral directions yields that even a small 
perturbation of $L$ may not be topologically conjugate to it. However, if a conjugacy between $L$ 
and a diffeomorphism $f$ exists, the question of its regularity  is interesting. As in the hyperbolic case, 
any two conjugacies  differ by an affine automorphism 
of $\T^d$ commuting with $L$, and so have the same regularity. 

Prior techniques of showing smoothness of a conjugacy along contracting and expanding foliations do not extend to the neutral foliations. Our methods, however, adapt well to the partially hyperbolic case, 
and we establish regularity of a conjugacy between $L$ and  $f$ without any closeness assumption. 

\begin{theorem} \label{PH}
Let $L$ be a very weakly irreducible partially hyperbolic ergodic automorphism of $\T^d$.
If a $C^\infty$ diffeomorphism $f$ of $\T^d$ is $C^{1+\text{H\"older}}$ conjugate to $L$, 
then any conjugacy between $f$ and $L$ is a $C^\infty$ diffeomorphism. 
\end{theorem}

To the best of our knowledge, this is the first bootstrap of regularity result for the partially hyperbolic case. 
In this theorem we assume $C^{1+\text{H\"older}}$ regularity of $H$, which is needed for our approach. 
In contrast to  the hyperbolic case, we do not know whether $C^1$ regularity implies 
$C^{1+\text{H\"older}}$. This lack of H\"older estimates is one of the reasons why KAM-type iterative argument in \cite{KSW} did not work in  the partially hyperbolic case.

\vskip.2cm

\subsection{Applications: improving regularity of conjugacy in prior rigidity results.} $\;$\\
Applying Theorems \ref{mainA} and \ref{PH} we improve the regularity of conjugacy 
from $C^{1+\text{H\"older}}$ to $C^\infty$ in local and global rigidity results for irreducible 
hyperbolic and partially hyperbolic toral automorphisms.  For a large class of toral automorphisms,
these corollaries describe diffeomorphisms which are $C^\infty$ conjugate to them.

The first result is global and is called $C^\infty$ periodic data rigidity of $L$. 
It follows from Theorem  \ref{mainA} and $C^{1+\text{H\"older}}$ periodic data rigidity,
which was recently obtained in \cite{DWG} as the globalization of the
corresponding local result in \cite{GKS11}. The only prior results on $C^\infty$ periodic 
data rigidity were obtained for automorphisms of $\T^3$ in \cite[Theorem 1.1]{DWG}
and for some automorphisms conformal on full stable and unstable bundles \cite{L02,L04,KS09}.

\begin{corollary} {\em(of Theorem  \ref{mainA} and \cite[Proposition 1.6]{DWG})} \label{PD} $\;$\\
Let $L:\T^d\to\T^d$ be an irreducible hyperbolic automorphism
such that no three of its eigenvalues have the same modulus.
Let  $f$ be a $C^\infty$ Anosov diffeomorphism of $\T^d$ in the homotopy class of $L$ such that the
derivative $D_pf^n$ is conjugate to $L^n$ whenever $p=f^n(p)$.
Then $f$ is $C^{\infty}$ conjugate to $L$.
\end{corollary}

We also obtain improvements in local Lyapunov spectrum rigidity results.

\begin{corollary} {\em (of Theorem  \ref{mainA} and  \cite[Theorem 1.1]{GKS20})} \label{local LS}$\;$\\
Let $L:\T^d\to\T^d$ be an irreducible hyperbolic automorphism such that no three of
its eigenvalues have the same modulus and there are no pairs of eigenvalues
of the form $\lambda, -\lambda$ or $i\lambda, -i\lambda$, where $\lambda\in \R$.
Let  $f$ be a volume-preserving $C^\infty$ diffeomorphism of $\T^d$ sufficiently
$C^{1}$-close to $L$. If the Lyapunov exponents of $f$ with respect to the volume
are the same as the Lyapunov exponents of $L$, then $f$ is $C^{\infty}$ conjugate to $L$.
\end{corollary}

A similar result in a partially hyperbolic setting uses stronger assumptions on the spectrum of $L$ 
and on closeness of $f$ to $L$. A toral automorphism $L$ is called {\em totally irreducible} if 
$L^n$ is irreducible for all $n\in \N$. Such $L$ is ergodic.

\begin{corollary} {\em (of Theorem  \ref{PH} and  \cite[Theorem 1.5]{GKS20})} \label{partial} $\;$\\
Let $L:\T^d\to\T^d$ be a totally irreducible automorphism with exactly two eigenvalues 
of modulus one and simple real eigenvalues away from the unit circle.
Let  $f$ be a volume-preserving $C^\infty$ diffeomorphism sufficiently $C^{N}$-close to $L$,
where $N=5$ if $d >4$ and $N= 22$ if $d=4$. 
If the Lyapunov exponents of $f$ with respect to the volume are the same as the Lyapunov 
exponents of $L$ then $f$ is $C^{\infty}$ conjugate to $L$.
\end{corollary}


\subsection{Proof strategy.}

We reduce the conjugacy equation \eqref{Conj def} to a  cohomological equation over $f$
with hyperbolic twist $L$
\begin{align*}\label{twist h}
 L\, h - h\circ  f=R ,
\end{align*}
where $h=H-I$ and $R=f-L$ can be viewed as functions from $\T^d$ to $\R^d$. Projecting this 
equation to a Lyapunov subspace $E^\rho \subset E^u$  for $L$ 
we obtain 
\begin{equation} \label{h rho}
 L_\rho h_\rho - h_\rho \circ f= R_\rho, \quad \text{where } L_\rho=L|_{E^\rho}.
\end{equation}
This is a fixed point equation on $h_\rho$  for an affine contraction on the space $C^0(\T^d,E^\rho)$. 
It follows that $h_\rho$ is the unique fixed points given by series 
\begin{equation} \label{series}
h_\rho =  \sum_{k=0}^\infty L_\rho^{-k-1} (R_\rho \circ f^k). 
\end{equation}
Term-wise differentiation can be used to show that $h_\rho$ is $C^\infty$ along the leaves of $\w^s$.
However, regularity of $h_\rho$  along $\w^u$ cannot be analyzed directly, as the series of derivatives diverge since $Df^k$ grows exponentially.
  Our approach relies on estimating distributional derivatives of $h_\rho$ along $\w^u$ 
and showing that weak derivatives of all orders are in $L^2(\T^d)$. 
We use some techniques developed in \cite{FKS} in the context of $\Z^k$-actions to study
regularity of conjugacy along the {\em neutral} foliation of $f$, rather than the unstable one.
To analyze regularity of $h_\rho$ along $\w^u$, instead of \eqref{series} we consider a different representation 
using {\em negative} iterates of $f$,
\begin{equation} \label{series-}
h_\rho=-\sum_{k=1}^\infty L_\rho^{k-1} (R_\rho \circ f^{-k}).
\end{equation}
However, this series does not converge in $C^0(\T^d,E^\rho)$. 
The key element of our approach is to differentiate it once along the corresponding fast 
subfoliation $\w$ of $\w^u$. Such derivatives $D^1_\w h_\rho$ converge in distributional sense as the contracting $Df^{-k}$ balances the expanding twist $L_\rho^{k-1}$. Then we differentiate further and show
that all distributional derivatives of $D^1_\w h_\rho$ along {\em all} of $\w^u$ are  in $L^2(\T^d)$.

Now we need to obtain weak differentiability of $h_\rho$ from that of $D^1_\w h_\rho$. 
This is the step where we use global information given by very weak irreducibility of $L$.
We recall that the bootstrap fails in general for reducible $L$. 
In particular, the regularity of  $h_\rho$ transverse to $\w$ cannot be recovered locally, even if  
$h_\rho$ is constant along $\w$.
In order to use global structure of $L$ to study nonlinear $f$,  we construct a {\em global} coordinate chart for the foliation $\w^u$ and the  
fast part $\w$. This allows us to use a suitable version of Diophantine property for the fast linear
foliation of $L$, which we obtain from very weak irreducibility. Using this Diophantine property
we establish a regularity result that allows us to recover sufficient information about all weak
derivatives of $h_\rho$ along $\w^u$  from that of $D^1_\w h_\rho$.

Our techniques are also well-adapted to treating the partially hyperbolic case, so that bootstrap from
$C^{1+\text{H\"older}}$ to $C^\infty$ requires only minor changes compared to hyperbolic one.




\section{Proof of Theorem \ref{mainA}} 

Since $L$ and hence $f$ are hyperbolic, we have hyperbolic splittings for $L$ and  $f$
\begin{align}\label{Hsplit}
\R^d=E^s\oplus E^u \quad \text{for $L$}\qquad \text{and } \quad \R^d=\E_x^s\oplus  \E_x^u \quad \text{for $f$}.
\end{align}
The subspace $E^{s/u}$ is the sum of all generalized eigenspaces of $L$ corresponding to 
eigenvalues of moduli less/grater than $1$.  The stable and unstable subbundles 
$\E^s$ and $\E^u$ of $T\T^d$ are tangent to $f$-invariant topological foliations $\w^s$ and $\w^u$
with {\em uniformly $C^\infty$ leaves}. This means that the individual leaves are $C^\infty$ immersed
manifolds with the derivatives of any order continuous on $T\T^d$.
For any continuous conjugacy $H$ as in \eqref{Conj def} we have 
\begin{align}\label{for:186}
 H(\w^{s})=W^{s}\; \text{ and }\;H(\w^{u})=W^{u},
\end{align}
where $W^s$ and $W^u$ are the linear foliations defined by $E^s$ and $E^u$.

\subsection{The conjugacy $H$ is $C^{1+\text{H\"older}}$} 
First we show that the conjugacy $H$ in Theorem \ref{mainA} is in fact $C^{1+\a}$ for some $\a>0$.
Since $H$ is $C^1$ we can differentiate the conjugacy equation $L\circ H= H \circ f$ and obtain 
$$L \circ DH(x)= DH(fx) \circ Df(x).
$$
Thus $C(x)=DH(x)$ is a continuous conjugacy between the derivative cocycles $L$ and $Df$, 
that is, 
$$L = C(fx) \circ Df(x) \circ C(x)^{-1}.
$$
We consider the Lyapunov splitting $\R^d=\oplus_{\rho\in \Delta} E_{\rho}$ for $L$, 
where $\Delta$ is the set of moduli  of eigenvalues of $L$ and $E_\rho$ is the sum of generalized eigenspaces of $L$ corresponding to the eigenvalues of modulus $\rho$.
Denoting $\E_\rho(x)= C^{-1}(x)E_{\rho}$ we obtain a continuous $Df$-invariant splitting $T\T^d=\oplus_{\rho\in \Delta} \E_{\rho}$ with the same expansion/contraction rates as for $L$. 
It is well known that such splitting is 
$\a$-H\"older with some exponent $\a>0$ which depends on $\Delta$. 
Then the restriction $C_\rho(x)=C(x)|_{\E_{\rho}(x)}$ is a continuous conjugacy between $\a$-H\"older
linear cocycles $L|_{E_{\rho}}$ and $Df|_{\E_{\rho}(x)}$ over $f$. Both cocycles have one Lyapunov exponent
$\log \rho$, and $L$ is clearly fiber bunched as the nonconformality $\|L^n\| \cdot \|L^{-n}\|$ grows at
most polynomially and hence is dominated by expansion/contraction along (un)stable manifolds.
Hence \cite[Theorem 2.1]{KSW} applies and yields that the conjugacy $C_\rho(x)$ is $\a$-H\"older for each $\rho$. It follows that $DH(x)=C(x)$ is $\a$-H\"older.

\subsection{General construction and properties of $H$} \label{for:116} 
 
While we do not assume that $H$ is close to the identity map, we may assume that $H$ 
is homotopic to the identity. Indeed, the induced map $H_*\colon \Z^d\to\Z^d$ on the first 
homology group of $\T^d$ is given by a matrix $A\in GL(d,\Z)$, which defines an automorphism 
of  $\T^d$. Replacing $f$ by $A\circ f \circ A^{-1}$ and $H$ by $H \circ A^{-1}$, we may 
assume that $H_*=\Id$. In particular, this yields that $f$ is in the homotopy class of $L$.
Since $L$ fixes $0$ we see that $f$ fixes $H^{-1}(0)$. Conjugating $f$ by the translation
$x\mapsto x+H^{-1}(0)$ we can also assume that $f(0)=L(0)=H(0)=0$ is a common fixed point for
$f$ and $L$. 

From now on we will assume that $H$ is in the homotopy class of the identity, $f$ is in the homotopy class of $L$, and they satisfy $H(0)=f(0)=0$. Under these assumptions, the conjugacy 
$H$ is unique and is given by the following construction. 
We lift $f$ and $H$ to $\R^d$ as 
$$\bar H=\Id+\bar h \quad\text{and}\quad \bar f=L+\bar R,$$
where $\bar h, \bar R:\R^d \to \R^d$ are $\Z^d$-periodic functions satisfying
$\bar R(0)=\bar h(0)=0$. The lifts satisfy the conjugacy equation 
$$
L\circ \bar H=\bar H\circ \bar f\quad\text{which yields}\quad L\circ \bar h - \bar h\circ \bar f=\bar R.
$$
 The latter equation projects to the  following equations for functions on $\T^d$
\begin{align}\label{for:78}
 L\circ h - h\circ  f=R \quad \text{or } \quad  h=L^{-1}\circ  h\circ  f+L^{-1} R,
\end{align}
where $h, R:\T^d \to \R^d$  satisfy $R(0)=h(0)=0$ and are as regular as $f$ and $H$ respectively.
Thus we have that $R$ is $C^{\infty}$ and $h$ is $C^{1+\a}$. 

Using the hyperbolic splitting $\R^d= E^s\oplus E^u$ for $L$ we define the stable and unstable 
projections $h^s,h^u,R_s,R_u$ of $h$ and $R$ respectively. Projecting the second equation in 
\eqref{for:78} to $E^u$ we obtain 
\begin{equation} \label{H_*}
h^u= L_u^{-1} (h^u\circ f)+L_u^{-1} R_u, \quad \text{where } L_u=L|_{E^u}.
\end{equation}
This is a fixed point equation for $h^u$ with the affine operator 
\begin{equation} \label{T*}
T_u (\psi )=L_u^{-1} (\psi \circ f)+L_u^{-1} R_u.
\end{equation}
Since $\|L_u^{-1}\|<1$, the operator $T_u$ is a contraction on the space $C^0(\T^d,E^u)$, 
and thus $h^u$ is its unique fixed point given by 
\begin{equation} \label{Hu}
h^u = \lim_{k\to \infty} T_u^k (0) = \sum_{k=0}^\infty L_u^{-k-1} (R_u \circ f^k).
\end{equation}
Similarly, one obtains the equation  for $h^s$ and the series $h^s=-\sum_{k=1}^\infty L_s^{k-1} (R_s \circ f^{-k}).$

\begin{lemma} \label{h^u s smooth}
The unstable component $h^u$ is uniformly $C^\infty$ along $\w^s$, that is, it has derivatives of all orders along the leaves of $\w^s$ which vary continuously on $\T^d$. Similarly, $h^s$  is uniformly $C^{\infty}$ along $\w^u$.
\end{lemma}
\begin{proof}
This follows from \eqref{for:186}. Specifically, $H(\w^u)=W^u$ means that the locally defined unstable component of the conjugacy
$H^u=\Id^u+h^u$ is locally constant along the leaves of $\w^s$ and hence is as regular as these leaves.
Thus $h^u$ is uniformly $C^\infty$ along $\w^s$. 

Alternatively, this can be established by term-wise differentiation of the series \eqref{Hu}
as norms of the derivatives of $f^k$ along $\w^s$ decay exponentially. This is clear for 
$Df^k|_{\E^s}$, and for higher order derivatives this is given by  \eqref{D^mfest} 
in Lemma \ref{D^mf^n} (i).
\end{proof}

\subsection{Lyapunov splitting and smoothness of $h^u$.}
We will focus on the unstable component and show that $h^u$  is $C^{\infty}$ on $\T^d$. 
For this we consider the splitting of the unstable bundles for $L$ and $f$ into Lyapunov subspaces.  
Let $1<\rho_1 <\dots <\rho_\ell$ be the distinct moduli of the unstable eigenvalues of $L$ and let 
\begin{equation} \label{splitL}
E^u=E^{u,L}= E^1 \oplus E^2 \oplus \dots \oplus E^\ell
    \end{equation} 
be the corresponding splitting of $E^{u}$, where $E^i$ is the direct sum of generalized eigenspaces
corresponding to the eigenvalues with modulus $\rho_i$. Similarly to the above, we project the conjugacy equation  to this splitting.
We let $L_i=L|_{E^i}$ and denote by $h_i$ and $R_i$ the $E^i$ components of $h$ and $R$ respectively.  
 Then \eqref{for:78} yields
\begin{align}\label{for:123}
  L_{i}h_{i}(x)-h_{i}(f(x))=R_{i}(x),\qquad x\in\T^d.
\end{align}
As above, this equation has a unique solution given by the similar series 
$$
h_i=\sum_{k=0}^\infty L_i^{-k-1} (R_i \circ f^k).
$$
However, this series is still not useful for studying the regularity of $h_i$  along $\w^u$ as differentiating along $\w^u$ yields diverging series.

The main part of the proof is analyzing regularity of individual $h_i$ along $\w^u$ and showing 
inductively that they all are $C^{\infty}$ on $\T^d$. The inductive process is given by the next theorem, 
which yields that $h^u$ is  $C^\infty$. The argument for $h^s$ is similar and hence $h$ and $H=\Id+h$ 
are also $C^\infty$. 

\begin{theorem}\label{th ind} Suppose $i \in \{1, \dots, \ell\}$. If $h_j$ is $C^\infty$ for all $1\leq j< i$, then $h_i$ is also $C^\infty$.
\end{theorem}

In the base case $i=1$ the assumption becomes vacuous.
We emphasize that the regularity of $h_i$ in the theorem is the global regularity on $\T^d$ rather than 
the regularity along the corresponding foliation.

\section{Proof of Theorem \ref{th ind}} 
\subsection{Outline of the proof}
We recall that the conjugacy $H$ is a $C^{1+\a}$ diffeomorphism. Using the Lyapunov splitting \eqref{splitL} for $L$ we obtain the corresponding  splitting for $f$
 \begin{equation} \label{for:7}
 \E^{u}=\E^{u,f}= \E^1 \oplus \E^2\oplus \dots \oplus \E^\ell
    \end{equation}
  into $\a$-H\"older  $Df$-invariant sub-bundles $\E^j=DH^{-1}(E^j)$  for $j=1, \dots, \ell$.
   
The bundles $\E^{i}$ are tangent to foliations $\w^{i}$, which are mapped by $H$ to the corresponding linear ones: $H(\w^j)=W^j$ for all $j=1, \dots, \ell$. In particular, each $\w^{j}$ is a $C^{1+\a}$ foliation. However, even its individual leaves are not more regular in general. This is the main reason why the regularity
of $H$ cannot be bootstrapped by studying its restrictions to $\w^{i}$.
Instead, we show global smoothness of the component $h_i$ on $\T^d$.

We will work with the fast subbundle $\E^{i,\ell}$ and  the corresponding foliations $\w^{i,\ell}$, where
\begin{align*}
 \E^{i,\ell}= \E^i \oplus \dots \oplus \E^\ell =T\w^{i,\ell}.
\end{align*}
We note that each $\w^{i,\ell}$ has uniformly $C^\infty$ leaves, moreover it gives a $C^\infty$ 
subfoliation of the leaves of $\w^u$. While this holds in general, in our case it follows from
Proposition \ref{chart}.

Our approach is to represent the first derivatives of  $h_i$ along $\w^{i,\ell}$ 
as series over the {\em negative} iterates of $f$. Using exponential mixing we show that these series 
converge as distributions as the ``expanding twist" by $L_i$ is now balanced by the contracting 
derivative of $f^{-1}$ along $\w^{i,\ell}$ . Then we further differentiate these series 
 in distributional sense along all directions in $\w^u$. 
We obtain estimates of such derivatives of all orders in terms of fractional Sobolev norms of test functions and then show that they are in $L^2(\T^d)$. This is done in Proposition \ref{L^2}. 

In Proposition \ref{chart} we construct an appropriate global coordinate chart which sends 
foliations $\w^{u}$ and $\w^{i,\ell}$ to the corresponding linear foliations $W^u$ and $W^{i,\ell}$ for $L$. 
This allows us to use the Diophantine property of $W^{i,\ell}=E^{i,\ell}$ to show that {\em all} derivatives
 of  $h_i$ along $\w^{u}$ are in $L^2(\T^d)$. This is done in Proposition \ref{Dioph}.
Since by Lemma \ref{h^u s smooth}  $h_i$ is also uniformly $C^\infty$ along $\w^s$, and 
 we will conclude that $h_i$ is  $C^\infty$ on $\T^d$ using \cite[Theorem 3]{L01}.

\subsection{Global charts}
Modifying the $C^{1+\a}$ conjugacy $H$, we now construct a global chart  $\Gamma_i$ for foliations $\w^u$ and $\w^{i,\ell}$ such that $\Gamma_i$ is $C^\infty$ along $\w^u$.

\begin{proposition}\label{chart} Suppose that $i \in \{1, \dots, \ell\}$ and that $h_j\in C^\infty (\T^d)$ for $1\leq j< i$.\\
Then there exists a $C^{1+\a}$ diffeomorphism $\Gamma_i$ of $\T^d$ such that
\begin{enumerate}
  \item\label{for:9} $\Gamma_i$ maps $\w^u$ to $W^u$  and is uniformly $C^\infty$ along $\w^u$;

  \item\label{for:10} $\Gamma_i$ maps $\w^{i,\ell}$ to $W^{i,\ell}$.
\end{enumerate}
\end{proposition}
We note that $\Gamma_i$ may not map $\w^{s}$ to $W^{s}$. 

If $i=1$ in the proposition, we have  $\w^{i,\ell}=\w^{1,\ell}=\w^u$ and
the assumption that $h_j$  is $C^\infty$ becomes vacuous.
The argument applies and gives a $C^{1+\a}$ diffeomorphism $\Gamma_1$  
which is just a smoothing of $H$ along the leaves $\w^u$.

\begin{proof} 
Recall that the  conjugacy $H=\Id+h$, where $h:\T^d\to \R^d$, is a $C^{1+\a}$ diffeomorphism. We 
write $h=(h^s,h_1,\dots,  h_{\ell})$ and smooth each $h_j$, $j\geq i$, while keeping the other components unchanged. More precisely, we define
\begin{align} \label{chart h}
\tilde h_\varepsilon =\big(h^s,h_1,\dots,  h_{i-1}, {s}_\varepsilon(h_i),\dots ,  {s}_\varepsilon(h_{\ell})\big),
\end{align}
where ${s}_\varepsilon(h_j) \in C^\infty (\T^d)$ is obtained using a standard smoothing by convolution operator.   
Since $\tilde h_\varepsilon \to h$ in $C^1$ as $\varepsilon \to 0$, we see that  $\Id+\tilde h_\varepsilon$ is  
$C^1$ close to the diffeomorphism $H$  for small $\varepsilon$. It follows that $\Id+\tilde h_\varepsilon$ is invertible and hence is a $C^{1}$ diffeomorphism. We fix such $\e>0$ and define
\begin{align} \label{chart G}
\Gamma_i=\Id+\tilde h_\varepsilon.
\end{align}
 Since $\tilde h_\varepsilon$ remains $C^{1+\a}$ on $\T^d$, both $\Gamma_i$ and  $\Gamma_i^{-1}$ are also $C^{1+\a}$. 
 
For $j\geq i$ components $(\tilde h_\varepsilon)_j={s}_\varepsilon(h_j)$ are $C^\infty$ by smoothing, and  for $1\le j <i$ components $(\tilde h_\varepsilon)_j =h_j$   are
 $C^\infty$  by the assumption.  
 The stable component $\tilde h_\varepsilon^s =h^s$  remains unchanged  and hence is 
 uniformly $C^\infty$ along  $\w^u$ by Lemma \ref{h^u s smooth}. Thus the whole 
  $\tilde h_\varepsilon$ is
uniformly $C^\infty$ along  $\w^u$, and hence so is $\Gamma_i$.

 Also since  $\tilde h_\varepsilon^s =h^s$,  the property $H(\w^u) = W^u$ ensures that $\Gamma_i(\w^u) = W^u$.
Similarly, $(\tilde h_\varepsilon)_j =h_j$, for $1 \le j <i$, and so  $H(\w^{i,\ell}) = W^{i,\ell}$ yields $\Gamma_i(\w^{i,\ell}) = W^{i,\ell}$.
\end{proof}


\subsection{Derivatives along $\w^u$}

We denote the dimension of $E^u$  by $d^u$ and we fix an orthonormal basis $x_1,\dots, x_{d^{u}}$ 
of $E^u$ such that $x_1,\dots, x_{\dim E^{i,\ell}}$ is a basis of $E^{i,\ell}$. 
We denote by $m=(m_1,\dots,m_{d^u})$ a multi-index with nonnegative integer components, 
and the corresponding derivative of a function $\omega $ on $\T^d$ by
\begin{align*}
 D^m_{E^u}(\omega)&=\partial_{x_1}^{m_1}\partial_{x_2}^{m_2}\cdots \partial_{x_{d^u}}^{m_{d^u}}\omega.
\end{align*}
The global foliation chart $\Gamma_i$ allows us to conveniently describe the derivatives 
along the foliations  $\w^u$ and $\w^{i,\ell}$ in a similar way:
\begin{align*}
D^m_{\w^u}\omega=D^m_{E^u}(\omega\circ \Gamma_i^{-1})=\partial_{x_1}^{m_1}\partial_{x_2}^{m_2}\cdots \partial_{x_{d^u}}^{m_{d^u}}(\omega\circ \Gamma_i^{-1}).
\end{align*}

The proof of Theorem \ref{th ind} has two main parts. The first one is the following proposition, 
which shows that all derivatives of $h_i$ along  $\w^u$ are in $L^2$, as long as the first 
derivative is taken along $\w^{i,\ell}$.

\begin{proposition}\label{L^2}  
Suppose that $i \in \{1, \dots, \ell\}$ and that $h_j$ are $C^\infty$ for $1\leq j< i$. \\
Then for every $1 \le k \le \dim E^{i,\ell}$  and every multi-index  $m=(m_1,\dots,m_{d^u})$
$$
D^m_{E^u}  \partial_{x_k} ( h_{i}\circ \Gamma_i^{-1}) \in L^2(\TT^d).
$$
   \end{proposition}
  We will prove this proposition in Section \ref{proof L^2}. The proof is dynamical and uses  derivative estimates and exponential mixing.
 \vskip.1cm     
The second part is Proposition \ref{Dioph} where we improve the conclusion of Proposition \ref{L^2} by showing that  {\em all} 
derivatives $D^m_{E^u} (h_{i}\circ \Gamma_i^{-1})$ are in $L^2(\T^d)$. 
    
\begin{proposition}\label{Dioph}
Suppose that  $\,D_{E^u}^m\partial_{x_k}( h_{i}\circ \Gamma_i^{-1})\in L^2(\TT^d)$ \,for every $1 \le k \le \dim E^{i,\ell}$  and every multi-index  $m=(m_1,\dots,m_{d^u})$.
Then
 $$D_{E^u}^m( h_{i}\circ \Gamma_i^{-1}) \in L^2(\TT^d)\quad\text{for every $m=(m_1,\dots,m_{d^u})$.}
 $$

\end{proposition}
We will prove this proposition in Section \ref{proof Dioph}.
The proof is analytical and relies on a Diophantine property of the linear foliation $E^{i,\ell}$,
which we obtain from the very weak irreducibility of $L$. 
This is the only place where we use the irreducibility assumption. 
\vskip.1cm

Combining Propositions \ref{L^2} and \ref{Dioph} we conclude that $D^m_{E^u} (h_{i}\circ \Gamma_i^{-1}) \in L^2(\TT^d)$ for all $m$, that is, all derivatives $D^m_{\w^u} h_{i}$ of $h_i$ along $\w^u$ are in $L^2(\TT^d)$. Recall that by Lemma \ref{h^u s smooth} the map $h^u$ and hence each component $h_{i}$ are uniformly $C^\infty$ along $\w^{s}$, that is, $D^m_{\w^s} h_{i}$ are continuous and so are also in $L^2(\TT^d)$.
Now we apply \cite[Theorem 3]{L01}. 
It yields that all derivatives of order $m$ of a function $\phi$ are   in $L^2(\TT^d)$
 if   $\phi$ has  derivatives of order $m$ in $L^2(\TT^d)$ along finitely many transverse 
 foliations, each with uniformly $C^\infty$ leaves and admitting foliation charts whose Jacobians are
 uniformly $C^\infty$ along the foliation. The last assumption holds for $\w^s$ and $\w^u$ \cite[Theorem 2]{L01}.
 Thus we conclude that all derivatives 
of $h_i$ are in $L^2(\TT^d)$, and hence $h_i \in C^\infty(\TT^d)$ by Sobolev embedding theorems.

This completes the proof of Theorem \ref{th ind} modulo the proofs of Propositions \ref{L^2} and \ref{Dioph}. For convenience of the exposition we will prove Proposition \ref{Dioph} first.


\section{Proof of Proposition \ref{Dioph}} \label{proof Dioph}
 The proof of Proposition \ref{Dioph} is structured as follows. 
In Section \ref{Weak irred} we give corollaries of very weak irreducibility 
and then state and prove a certain Diophantine property for spaces $E^{i,\ell}$.
In Section \ref{functions} we discuss fractional Sobolev spaces $\mathcal{H}^\b$ 
that we will use in our regularity results.  In Section \ref{dioph result} we state
and prove the main technical result, Proposition \ref{dioph}, and in Section~\ref{finish} 
we use it to complete the proof of Proposition \ref{Dioph}.

\subsection{Very weak irreducibility and Diophantine property.} \label{Weak irred}
For a matrix $L\in GL(d,\Q)$ we denote the largest modulus of its eigenvalues by 
$\rho_{\max}$. Let $E_{\max}$ be the direct sum of its 
generalized eigenspaces corresponding to the eigenvalues with modulus $\rho_{\max}$.
We denote by $\hat E_{\max}$ the direct sum of its generalized eigenspaces corresponding 
to the eigenvalues with modulus different from $\rho_{\max}$. Thus we have  an  $L$-invariant 
splitting $\R^d=E_{\max}\oplus \hat E_{\max}$. We also denote by $(E_{\max})^\perp$
the orthogonal complement of $E_{\max}$ with respect to the standard inner product. 

For be the characteristic polynomial  $p$ of $L$ we consider its prime decomposition over $\Q$
$$p(t)=\prod_{k=1}^K (p_k(t))^{d_k}$$
and  the corresponding splitting of $\R^d$ into $L$-invariant  rational subspaces
$$
 \R^d = \oplus V_k \, , 
 \quad \text{where} \quad V_k =\ker\, (p_k^{d_k}(L)).
$$

 \begin{lemma}  \label{weak irred}
For any matrix $L\in GL(d,\Q)$ the following are equivalent.
\begin{enumerate}
\item Each $p_k$ has a root of modulus $\rho_{\max}$,
\item $\hat E_{\max} \cap \Z^d=\{0\}$,
\item $(E_{\max})^\perp\cap \Z^d=\{0\}$.
\end{enumerate}
\end{lemma}
Since we focus on $\w^u$ we deal only with $\rho_{\max}$.  A similar result holds for  $\rho_{\min}$, 
the smallest modulus  of eigenvalues of $L$, and the corresponding spaces $(E_{\min})^\perp$ and $\hat E_{\min}$. It would be used in the proof of smoothness of $h^s$.
 \begin{proof}

(2) $\Rightarrow$ (1)\, If some  $p_k$ has no roots of modulus $\rho_{\max}$, then $V_k \subset \hat E_{\max}$. Since $V_k$ is a rational subspace, it contains nonzero points of $\Z^d$ and hence so does  $\hat E_{\max}$.
 
 (1) $\Rightarrow$ (2)\,
Suppose there is $0\ne n \in (\Z^d \cap \hat E_{\max})$. Then for some $k$ the component $n_k$ of $n$ in $V_k$ is nonzero and rational. We note that $n_k \in \hat E_{\max}$ as $\hat E_{\max}= \oplus_k (\hat E_{\max} \cap V_k)$. Then
 $$W=span \{ L^m n_k : m\in \Z \}
 $$
 is a rational $L$-invariant subspace contained in $\hat E_{\max} \cap V_k$. Then the characteristic 
 polynomial of $L|_W$ is a power of $p_k$ and hence has a root of modulus $\rho _{\max}$ by (1). Thus
 $W \cap E_{\max} \ne 0$,  contradicting $W \subset \hat E_{\max}$. 
 
 \vskip.1cm
 (1) $\Leftrightarrow$ (3)\,
 Since the transpose $L^\tau$ has the same  characteristic polynomial $p$, (1) is the same as  the
 corresponding property (1$^\tau$) for  $L^\tau$, and hence it is equivalent to (2$^\tau$) with
  the corresponding subspace for  $L^\tau$: $\hat E_{\max}^\tau \cap \Z^d=\{0\}$. It remans to note
  that $(E_{\max})^\perp=\hat E_{\max}^\tau$.
 Indeed, the polynomial
$$q(x)=\prod _{|\lambda |=\rho_{\max}}(x-\lambda)^d,
$$ 
where the product is over all eigenvalues of $L$ of modulus $\rho_{\max}$, is real and we obtain
\begin{align} \label{LA}
(E_{\max})^\perp=(\ker q(L))^\perp = \text{range}\, (q(L))^\tau=  \text{range}\, (q(L^\tau)) =\hat E_{\max}^\tau
\end{align}
since 
 $q(L^\tau)$ is invertible on $\hat E_{\max}^\tau$ and zero on $E_{\max}^\tau$. 
\end{proof}

\begin{definition} We say that a subspace $V$ of $\R^d$ has {\em Diophantine property} 
if there exists $K>0$ such that for any $n\in\Z^d$ 
and any orthonormal basis $\{v_1,\dots, v_{\dim V}\}$ of $V$ we have
\begin{align} \label{diophantine}
 \sum_{i=1}^{\dim V}|n\cdot v_i|\geq K \norm{n}^{-d},
\end{align}
where $\|n\|$ is the standard norm of $n$ and \,$n\cdot v_i$ is the standard inner product in $\R^d$.
\end{definition}

\begin{lemma} The space $E^{i,\ell}$ has Diophantine property for each $i \in \{1, \dots, \ell\}$,
\end{lemma}

\begin{proof}

It suffices to prove the lemma for $i=\ell$ as it gives the smallest subspace $E^{\ell}$. 
 For any orthonormal basis $\{v_1,\dots, v_{\dim E^{\ell}}\} $ of $E^{\ell}$ we have 
$$\sum_{j=1}^{\dim E^{\ell}}|n\cdot v_{i}|  \ge \norm{\pi_{E^{\ell}} (n)} =\dist (n,(E^{\ell})^\perp),
 $$
where $ \pi_{E^{\ell}}$ denotes the orthogonal projection to $E^{\ell}$. To complete the proof
we need to show that $\dist (n,(E^{\ell})^\perp)\geq K \norm{n}^{-d}$. Here we use Katznelson's Lemma,  see e.g. \cite[Lemma 4.1] {Damjanovic4} for a proof.

\begin{lemma}[Katznelson's Lemma]\label{Katz} Let $A$ be a $d\times d$ integer matrix.
Assume that $\R^d$ splits as $\R^d=V_1\oplus V_2$, where $V_1$ and $V_2$ are invariant under $A$,
and  $A|_{V_1}$ and $A|_{V_2}$ have no common eigenvalues. If $V_1\cap\Z^d=\{0\}$, then
there exists a constant $K$ such that
\begin{align*}
\dist (n,V_1)\geq K\norm{n}^{-d} \quad\text{for all } \,0\ne n\in\Z^d,
\end{align*}
where $\norm{n}$ denotes
Euclidean norm and $dist$ is Euclidean distance.
\end{lemma}

We apply this lemma with the matrix $A=L^\tau$ and its invariant splitting 
$\R^d=\hat E_{\max}^\tau\oplus  E_{\max}^\tau$. We use the notations of 
Lemma \ref{weak irred},  so that  $E^{\ell}=E_{\max}$ and by \eqref{LA} we have
$(E_{\max})^\perp=\hat E_{\max}^\tau$. By very weak irreducibility of $L$ and 
Lemma \ref{weak irred}(3) we have $(E_{\max})^\perp\cap \Z^d=\{0\}$. Thus the assumptions of 
Katznelson's Lemma are satisfied and it yields $\dist (n,(E_{\ell})^\perp)\geq K \norm{n}^{-d}$
as desired.
\end{proof}


\subsection{Fractional Sobolev spaces.} \label{functions}
We will use  fractional Sobolev spaces  $\mathcal{H}^\b$ on $\T^d$ which can be defined in terms of 
Fourier coefficients as follows. For any function $\omega \in L^2(\T^d,\CC)$ we denote its 
Fourier coefficients by $\hat \omega_n$, $n\in\Z^d,$ and write its Fourier series   
 $$ \omega(x)=\sum_{n\in\Z^d}\widehat{\omega}_ne^{2\pi \mathbf{i} \, n\cdot x}.
 $$
For any $\b > 0$ we define the norm
\begin{align}\label{for:181}
 \norm{\omega}_{\mathcal{H}^\b}=\big(\sum_{n\in\Z^d}(1+\norm{n}^2)^{\b} \cdot |\hat \omega_n|^2 \big )^{1/2}
 \,\ge\, \norm{\omega}_{L^2}
\end{align}
and the fractional Sobolev  space 
\begin{align}\label{Halpha}
\mathcal{H}^\b =\{  \omega \in L^2(\T^d):\;  \norm{\omega}_{\mathcal{H}^\b}<\infty\}
\quad \text{and } \; \; \mathcal{H}^\b_0=\{  \omega \in \mathcal{H}^\b:\; \hat  \omega_0=0\}.
\end{align}
For $k\in \N$, the space $\mathcal{H}^k$ coincides with the usual Sobolev space $W^{k,2}$ of $L^2$ functions whose weak derivatives of order up to $k$ are in $L^2$.
By the Sobolev embedding theorem,  for any $k,r \in \N$ such that  $k>r+d/2$ we have 
  \begin{align}\label{for:124}
  \mathcal{H}^k \subset C^r  \quad \text{and } \quad \norm{\omega}_{C^r}\leq M \norm{\omega}_{\mathcal{H}^k} .
  \end{align}
We will work with $\mathcal{H}^\b$ for $0<\b<1$ and use the inclusion of the space of $\a$-H\"older functions
 \begin{align}\label{Holder}
  C^\a \subset \mathcal{H}^\b   \quad \text{for any $0<\b <\a$ }.
  \end{align}
This can be easily seen by using the H\"older estimate in the numerator of
the norm 
$$
 (\norm{\omega}'_{\mathcal{H}^\b})^2=\int_{\T^d} \int_{\T^d} \frac{|\omega(x)-\omega(y)|^2}{|x-y|^{d+2\b}}\,dx\, dy.
$$
which is equivalent to the norm \eqref{for:181}, see e.g. \cite{BO}.

\subsection{Diophantine regularity result.} \label{dioph result} 

Now we state the main analytical result which uses the Diophantine property. 
It relates differentiability of a function to that of its first derivatives along a foliation $V$ 
with the Diophantine property. Here we consider distributional derivatives. For a function 
$\omega \in L^2(\T^d)$ and a $C^\infty$ test function $\psi:\T^d\to \C$ we denote their pairing as 
\begin{align}\label{pairing}
 \langle \omega, \psi\rangle= \int_{T^d} \omega (x) \, \bar  \psi (x)\,dx,
\end{align}
where the integral is with respect to the Lebesgue measure.

For a multi-index  $m=(m_1,\dots,m_{\dim E})$ the distributional derivative 
$D^m_{E}\omega$ of $\omega$ is defined as the functional on the space of $C^\infty$ test functions by
\begin{align}\label{D der}
 \langle D^m_{E}\omega, \psi\rangle= (-1)^{|m|} \big\langle \omega   , D^{m}_{E} \psi\big\rangle ,
  \quad \text{ where } \; |m|=\sum _{k=1}^{\dim E} {m_k}. 
\end{align}
We write that 
$D^m_{E}\omega\in L^2(\TT^d)$ if this distribution is given by an $L^2$ function.


\begin{proposition}\label{dioph}
Let  $V$ be a subspace in  $\RR^d$ with Diophantine property and let $E$ be any subspace of 
$\R^d$. Let $\{v_1,\dots, v_{\dim V}\}$  and  $\{e_1,\dots, e_{\dim E}\}$ be  their orthonormal  bases. 
Suppose that $\omega\in L^2(\TT^d)$ satisfies 
$D^m_{E}\partial_{v_j}\omega\in L^2(\TT^d)$ \,for every $1 \le j \le {\dim V}$  and every multi-index  $m=(m_1,\dots,m_{\dim E})$.
 Then for any $0<\b<1$ and every multi-index  $m=(m_1,\dots,m_{\dim E})$
there exists a constant $K=K(d,m,\b,V,\omega)$ such that 
\begin{align}\label{for:27}
 |\langle D^m_{E}\omega, \psi\rangle|=|  \big\langle \omega   , D^{m}_{E} \psi\big\rangle |\leq K\norm{\psi}_{\mathcal{H}^\b}\quad \text{ for every $C^\infty$ test function $\psi$}.
\end{align}
\end{proposition}

\begin{proof}

For any $C^\infty$ function $\omega$, multi-index  $m=(m_1,\dots,m_{\dim E})$, and $n\in \Z^d$ we have 
 \begin{align} \label{FC of D}
 (\widehat{D^{m}_{E}\omega})_n=  \langle D^{m}_{E}\omega, e^{2\pi \mathbf{i} \, n\cdot x}\rangle= 
 (-1)^{|m|} \langle \omega, D^{m}_{E} e^{2\pi \mathbf{i} \, n\cdot x}\rangle=
 (-1)^{|m|} \,\overline{(2\pi \mathbf{i})^{|m|}n^{m} }\; \widehat{\omega}_n
  \end{align}
and hence 
\begin{align} \label{FC of Dn} \quad |(\widehat{D^{m}_{E}\omega})_n|=  
 (2\pi )^{|m|}|n^{m}| \; |\widehat{\omega}_n| \quad\text{ where }\quad n^m=\prod _{k=1}^{\dim E} (n\cdot e_k)^{m_k}.
  \end{align}
These equalities hold for any $L^2$ function $\omega$ by the definition
of distributional derivative.

First we express the assumption $D^m_{E}\partial_{v_j}\omega\in L^2(\TT^d)$ in terms of Fourier coefficients
using  $\,|\widehat {\partial_{v}\omega}_n|=2\pi|n\cdot v| \, | \hat {\omega}_n| $ and \eqref{FC of Dn}
\begin{align*}
 \sum_{n\in\Z^d}|\hat \omega_n|^2\,|n^{m}|^2\, |n\cdot v_j|^2= (2\pi)^{-2(|m|+1)}\| D^m_{E}\,\partial_{v_j}\omega \|_{L^2}^2<\infty.
\end{align*}
Using Cauchy-Schwarz inequality we obtain
\begin{align*}
 &\sum_{n\in\ZZ^d}|\hat \omega_n|^2\,|n^{m}|\,|n\cdot v_j|=\sum_{n\in\ZZ^d}(|\hat \omega_n|\, |n^{m}|\,|n\cdot v_j|)\, |\hat \omega_n|
  \\
 &\leq \big( \sum_{n\in\ZZ^d}|\hat \omega_n|^2\, |n^{m}|^{2}\,|n\cdot v_j|^2\big)^{1/2} 
 \cdot         \big( \sum_{n\in\ZZ^d}|\hat \omega_n|^2\big)^{1/2}\le\,
 \| D^m_{E}\,\partial_{v_j}\omega \|_{L^2} \cdot \| \omega \|_{L^2}<\infty.
\end{align*}
Using this inductively to divide exponent of $|n^{m}|\,|n\cdot v_j|$ by $2$ we obtain that for any $m$ and any $k\in \N$ 
\begin{align*}
  \sum_{n\in\ZZ^d}|\hat \omega_n|^2\, |n^{m}|^{2/2^k}\,|n\cdot v_j|^{2/2^k}=
  \sum_{n\in\ZZ^d}(|\hat \omega_n|\, |n^{m}|^{2/2^k}|n\cdot v_j|^{2/2^k})\,  |\hat \omega_n| \le \\
\le \big( \sum_{n\in\ZZ^d}|\hat \omega_n|^2\, |n^{m}|^{2/2^{k-1}}|n\cdot v_j|^{2/2^{k-1}} \big)^{1/2} 
\cdot \| \omega \|_{L^2} \,\le\, K_1 (k,m,   \omega ) <\infty.
\end{align*}
Since $m$ is arbitrary, taking $m=2^k m'$ we can rewrite it as 
\begin{align}\label{for:3}
 \sum_{n\in\Z^d}|\hat \omega_n|^2\,|n^{m'}|^{2}\,|n\cdot v_j|^{2/2^k}\le\,
K_1 (k,2^km',   \omega )  \quad \text{ for all $k\in\N$ and all $m'$.}
\end{align}
Informally, the last  inequality means that  $D_{E}^{m'}\partial^{1/2^k}_{v_j}\omega\in L^2(\TT^d)$. More precisely, for any $\theta \in \mathcal{H}^\b(\TT^d)$ we  define the following ``fractional derivative"
\begin{align}\label{frac der}
 |\partial_{v_j}|^{\b}\,\theta\,=\sum_{n\in\ZZ^d} \widehat{\theta}_n\,|n\cdot v_j|^\b\, e^{2\pi \mathbf{i} \, n\cdot x}\;  \in L^2(\TT^d).
\end{align}
 To prove \eqref{for:27} we will use \eqref{for:3} along with a representation of the test function $\psi$ as a sum of suitable   fractional derivatives. The latter is given by the next lemma, which relies on the Diophantine property of $V$.

\begin{lemma} \label{lemma theta}
 Suppose that $0<\b<1$ and $k\in \NN$ satisfy $\,d/2^k<\b/2$. Then there exists $K_3=K_3(V,k)$
 such that for any 
$\psi\in \mathcal{H}^\b$  there exist 
\begin{align}\label{for:255}
\theta_j\in \mathcal{H}^{\b/2}_0\;\text{ with }\;\norm{\theta_j}_{\b/2}\leq K_2\norm{\psi}_{\mathcal{H}^\b},\quad 1\leq j\leq \dim V, \quad \text{ such that}
\end{align}
\begin{align}\label{for:26}
  \psi=\hat \psi_0+\sum_{j=1}^{\dim V} |\partial_{v_j}|^{1/2^k} \theta_j.
\end{align}
\end{lemma}
\begin{proof}
\,\,For each $0\ne n \in \ZZ^d$ we define $\iota(n) \in \{1,\dots, \dim V\}$ to be the smallest index for which
$|n\cdot v_{\iota(n)}|$ is maximal, and hence 
$$|n\cdot v_{\iota(n)}|\,\geq\, \text{{\small$\frac{1}{\dim V}$}}\sum_{i=1}^{\dim V}|n\cdot v_i|.
$$ 
Then we can write $\psi=\hat \psi_0+\sum_{j=1}^{\dim V}\psi_j$,\, where the function 
  $\psi_j = \sum_{0\ne n\in\Z^d}\widehat{\psi}_n\,e^{2\pi \mathbf{i} \, n\cdot x}$
is defined by
\begin{align*}
 (\widehat{\psi_j})_n=\widehat{\psi}_n\; \text{ if }\,\iota(n)=j \;\, \text{ and otherwise }(\widehat{\psi_j})_n=0.
\end{align*}
Now for each $ 1\leq j\leq \dim V$ we construct 
$\theta_j = \sum_{0\ne n\in\Z^d}\widehat{\theta}_n\,e^{2\pi \mathbf{i} \, n\cdot x}$ satisfying
\begin{align*}
  |\partial_{v_j}|^{1/2^k} \theta_j=\psi_j \quad \text{by taking} \quad
  (\widehat{\theta_j})_n=  |n\cdot v_j|^{-1/2^k}(\widehat{\psi_j})_n\quad \forall \; 0\ne n\in\ZZ^d.
\end{align*}
We note that $(\widehat{\theta_j})_n=0$ unless $\iota(n)=j$, in which case we have
\begin{align*}
|n\cdot v_{j}|\,\geq\, \text{{\small$\frac{1}{\dim V}$}}\sum_{i=1}^{\dim V}|n\cdot v_i| \,\geq\, 
K_2(V)\norm{n}^{-d}
\end{align*}
by the Diophantine property \eqref{diophantine} of $V$. Thus every nonzero 
Fourier coefficient of $\theta_j$ can be estimated as
\begin{align*}
 |(\widehat{\theta_j})_n|=\frac{|(\widehat{\psi_j})_n|}{|n\cdot v_j|^{1/2^k}}\, \leq\, |(\widehat{\psi_j})_n|\,K_3(V,k)\,\norm{n}^{d/2^k}
 =\, K_3(V,k) \,|\widehat{\psi}_n|\, \norm{n}^{d/2^k}.
\end{align*}
Since $\psi \in  \mathcal{H}^{\b}$, see \eqref{for:181} and \eqref{Halpha}, it follows that  $\theta_j\in \mathcal{H}^{\b-(d/2^k)}$ and 
$\norm{\theta_j}_{\b-(d/2^k)}\leq C\norm{\psi}_{\mathcal{H}^\b}$. Since $\frac{d}{2^k}<\frac{\b}{2}$ 
we have  $\b-\frac{d}{2^k}>\frac{\b}{2}$,\, and hence   $\theta_j$ satisfies \eqref{for:255}.
\end{proof}

Now we finish the proof of Proposition \ref{dioph} using Lemma \ref{lemma theta}. 
We consider $\psi\in \mathcal{H}^\b$ and choose  $k=k(d,\b)$  sufficiently large 
so that the lemma  applies. Then we can estimate
\begin{align*}
 |\langle D^m_{E}\omega, \psi\rangle| & \,\overset{\text{(1)}}{\leq}\; 
 \sum_{j=1}^{\dim V} \Big| \big\langle D^m_{E}\omega,\, 
 |\partial_{v_j}|^{1/2^k} \theta_j\big\rangle \Big| 
\,\overset{\text{(2)}}{\leq}\, \sum_{j=1}^{\dim V} \sum_{n\in\ZZ^d}(2\pi)^{|m|}|\hat \omega_n|\,
   |n^{m}|\, |n\cdot v_j|^{1/2^k} |(\hat \theta_j)_n|\\
 &\overset{\text{(3)}}{\leq}\,(2\pi)^{|m|} \sum_{j=1}^{\dim V}
 \big( \sum_{n\in\ZZ^d}|\hat \omega_n|^2\,
  |n^{m}|^{2}\,|n\cdot v_j|^{2/2^k}\big)^{1/2}\cdot 
 \big( \sum_{n\in\ZZ^d}|(\hat{\theta_j})_n|^2\big)^{1/2}\\
 &\overset{\text{(4)}}{\leq} (2\pi)^{|m|}\, 
 \big(K_1 (k,2^km,   \omega )\big)^{1/2}
 \sum_{j=1}^{\dim V} \norm{\theta_j}_{L^2} \\
&\overset{\text{(5)}}{\leq}\;  (2\pi)^{|m|} \,
\big(K_1 (k,2^km,   \omega )\big)^{1/2}
(K_3 \dim V)\,\norm{\psi}_{\mathcal{H}^\b}=K(\b,m,d,V,\omega)\,\norm{\psi}_{\mathcal{H}^\b}.
\end{align*}
Here in $(1)$ we use \eqref{for:26} and $\langle D^m_{E}\omega, \hat \psi_0\rangle=0$, 
in $(2)$ we use \eqref{FC of Dn} and \eqref{frac der},
in $(3)$ we use Cauchy-Schwarz inequality, in $(4)$ we use \eqref{for:3}, 
and in $(5)$ we use \eqref{for:255}.

\vskip.1cm
This completes the proof of Propositions \ref{dioph}. 
\end{proof}

\subsection{Completing the proof of Proposition \ref{Dioph}} \label{finish}

We fix coordinates in $E^i$ and we write the vector-valued function $\omega=h_{i}\circ \Gamma_i^{-1} :\T^d \to E^i$ 
as $(\omega_1, \dots , \omega_{\dim E^i})$. We apply 
Proposition \ref{dioph}  with 
$$V=E^{i,\ell}, \quad E=E^u,\;\text{ and }\; \omega_j, \;\;j=1, \dots , \omega_{\dim E^i}.
$$
Since $h_{i}\circ \Gamma_i^{-1}$ is $\a$-H\"older continuous, all $\omega_j$ are in $\mathcal{H}^\b$
for any $\b<\a$ by \eqref{Holder}. 

Now  Proposition \ref{Dioph} follows from the next lemma, which upgrades the derivatives to $L^2$.
While the statement is natural, we do not know if it holds with $\mathcal{H}^\b$ replaced by $C^\b$.

\begin{lemma} \label{to L^2}
 Suppose that $\omega \in \mathcal{H}^\b$ and that for some multi-index  $m$
 and $ K>0$ we have
 \begin{align} \label{phi est}
|  \big\langle D^{2m}_{E}\omega , \psi \big\rangle |
 \le K  \norm{\psi}_{\mathcal{H}^\b}
  \end{align}
for every $C^\infty$ test function $\psi$. Then $D^m_{E}\omega \in L^2(\T^d)$.
\end{lemma}

\begin{proof}

We consider the smoothing  of $\omega$ by truncation of its Fourier series, 
\begin{align} \label{trunc}
\omega_N=\sum_{n\in\ZZ^d,\, \|n\|\leq N} \widehat{\omega}_ne^{2\pi \mathbf{i} \, n\cdot x}. 
 \end{align}
Each $\omega_N$ is $C^\infty$
and satisfies $\| \omega_N\| _{\mathcal{H}^\b} \le \|\omega \| _{\mathcal{H}^\b}$. 
Then  \eqref{phi est} and \eqref{FC of Dn} yield
\begin{align*}
K  \norm{\omega}_{\mathcal{H}^\b} &\,\ge\, K  \norm{\omega_N}_{\mathcal{H}^\b} \,\ge \,
| \big\langle D^{2m}_{E}\omega, \, \omega_N \big\rangle |\,=\,
|  \big\langle \omega, \, D^{2m}_{E}\omega_N \big\rangle | \,=\, 
|\int_{\T^d} \omega (x) \,  \overline{D^{2m}_{E}\omega_N(x)}\,dx\,|\\
&=|\sum_{\|n\|\leq N} \widehat{\omega}_n \, \overline{(\widehat{D^{2m}_{E}\omega})_n} \,|\,=\,
|\sum_{\|n\|\leq N} \widehat{\omega}_n \, (2\pi \mathbf{i})^{2|m|}\,n^{2m} \,  \overline{\widehat{\omega}_n} \,| \\
&=\sum_{\|n\|\leq N}  (2\pi)^{2|m|}\,n^{2m}\, |\widehat{\omega}_n|^2
\,=\,\sum_{\|n\|\leq N}  \left( (2\pi)^{|m|}|n^{m}|\, |\widehat{\omega}_n|\right)^2\quad \text{ for any $N\in \N$.}
\end{align*}
 We conclude using \eqref{FC of Dn} that
$$\sum_{n\in\ZZ^d}  |(\widehat{D^{m}_{E}\omega})_n|^2 = \sum_{n\in\ZZ^d}  \left( (2\pi)^{|m|}|n^{m}|\, |\widehat{\omega}_n|\right)^2
\le K  \norm{\omega}_{\mathcal{H}^\b}$$
and hence the distribution $D^{m}_{E}\omega$ is given by the $L^2$ function 
$\sum_{n\in\ZZ^d} (\widehat{D^{m}_{E}\omega})_n\,e^{2\pi \mathbf{i} \, n\cdot x}$.
\end{proof}

\section{Proof of Proposition \ref{L^2}} \label{proof L^2}



We prove the proposition in Section 5.1. In Section 5.2 we state and prove derivative estimates that are used in this proof as well as in Section \ref{PHproof}.
\subsection{Main part of the proof.}
For a vector-valued function $f=(f_1, \dots, f_N) :\TT^d\to \R^N$ we define its norm 
as the maximum $\norm{f}=\max_{i}\norm{f_i}$ of the corresponding norms of the components.
We also adopt vector-valued notations for the inner product in $L^2(\T^d)$ with respect to the Lebesgue 
measure and for pairing of a distribution with a test function $\psi :\T^d\to \R$
\begin{align}\label{for:198}
\langle f , \psi \rangle:=\big(\langle f_1, \psi \rangle, ,\cdots, \langle f_N, \psi\rangle \big)\in \R^N.
\end{align}


First, we obtain a suitable series representation for the distribution $D^m_{E^u}\partial_{x_j}(h_{i}\circ \Gamma_i^{-1})$, where $1\leq j\leq \dim E^{i,\ell}$. Specifically, we will show that
\begin{align} \label{hi series}
   \big\langle h_{i}\circ \Gamma_i^{-1}, \, D^m_{E^u}\partial_{x_j}\psi \big\rangle\,=\,
   \sum_{k=1}^\infty L_i^{k-1}\big\langle R_i\circ f^{-k}\circ \Gamma_i^{-1}, \, D^m_{E^u}\partial_{x_j}\psi\big\rangle.
  \end{align}
As we discussed in the introduction, this is different from differentiating series \eqref{series} for $h_{i}$ and instead corresponds to  differentiating series \eqref{series-}. However, \eqref{series-} does not converge
in $C^0$ and even its distributional convergence is a priory not clear.  In contrast, the series  \eqref{hi series} converges distributionally. 
To show this  we work directly with a finite iteration of the conjugacy equation  \eqref{for:123}. Rewriting  $L_{i}h_{i}(x)-h_{i}(f(x))=R_{i}(x)$ we obtain
\begin{align*}
 h_{i}(x)=L_{i}h_{i}(f^{-1}(x))-R_{i}(f^{-1}(x)) \quad \text{for all } \; x\in\T^d.
\end{align*}
Iterating this, we see that for any $n\in \NN$,
\begin{align*}
 h_i=L_i^{n}h_i\circ f^{-n}-\sum_{k=1}^nL_i^{k-1}R_i\circ f^{-k},
\end{align*}
and hence
\begin{align*}
 h_i\circ \Gamma_i^{-1}\,=\,L_i^{n}h_i\circ f^{-n}\circ \Gamma_i^{-1}-\sum_{k=1}^nL_i^{k-1}R_i\circ f^{-k}\circ \Gamma_i^{-1}.
\end{align*}
Now we show that the error term $L_i^{n}h_i\circ f^{-n}\circ \Gamma_i^{-1}$ paired with $D^m_{E^u}\partial_{x_j}\psi$ tends to zero  as $n \to \infty$. Recalling that $\Gamma_i$, $\Gamma_i^{-1}$,
$H$, $H^{-1}$, and $h_i$ are $C^{1+\a}$ we obtain

\begin{align}\label{for:19}
 &\qquad\; L_i^{n}\,\big\langle h_i\circ f^{-n}\circ \Gamma_i^{-1}, \, D^m_{E^u}\partial_{x_j}\psi \big\rangle\\
 &\overset{\text{(1)}}{=}-L_i^{n}\,\big\langle \partial_{x_j}(h_i\circ f^{-n}\circ \Gamma_i^{-1}), \, D^m_{E^u}\psi \big\rangle\notag\\
 &\overset{\text{(2)}}{=}-L_i^{n}\,\big\langle \partial_{x_j}(h_i\circ H^{-1}\circ L^{-n} \circ H \circ \Gamma_i^{-1}), \, D^m_{E^u}\psi \big\rangle\notag\\
 &=-L_i^{n}\,\big\langle D_{E^{i,\ell}}(h_i\circ H^{-1})_{L^{-n}\circ H\circ \Gamma^{-1}x}\cdot L^{-n}|_{E^{i,\ell}}\cdot \partial_{x_j}(H\circ \Gamma_i^{-1})_x, \, D^m_{E^u}\psi \big\rangle\notag\\
 &\overset{\text{(3)}}{=}-\rho_i^{-n}L_i^{n}\,\big\langle D_{E^{i,\ell}}(h_i\circ H^{-1})_{L^{-n}y}\cdot B_n\cdot \partial_{x_j}(H\circ \Gamma^{-1})_{\Gamma_i\circ H^{-1}y},\,D^m_{E^u}(\psi)_{\Gamma_i\circ H^{-1}y}\cdot J(y) \big\rangle \notag
\end{align}
Here in $(1)$ we integrate by parts, noting that the left side is $C^{1+\a}$ and so the pairing  is given by integration; in $(2)$ we use that $H$ conjugates $L^{-n}$ and $f^{-n}$; in $(3)$ we denote $B_n=\rho_i^{n}L^{-n}|_{E^{i,\ell}}$, change the variable to $y=H\circ \Gamma_i^{-1}(x)$,  and
 denote by $J(y)$ the Jacobian of the coordinate change. 


To estimate \eqref{for:19} we will use exponential mixing of H\"older functions with respect to the volume for ergodic automorphisms of tori. It was obtained by Lind \cite[Theorem 6]{Lind} and  by Gorodnik and Spatizer \cite[Theorem 1.1]{Gorodnik} for  nilmanifolds, which use for the estimate.
\begin{theorem}\label{th:1} \cite{Lind,Gorodnik} 
Let $L$ be an ergodic automorphism of $\TT^d$. Then for any $0<\a<1$ and any $g_1,\,g_2\in C^\a(\TT^d)$ there exists $\gamma=\gamma(\a)>0$ and $K=K(\a)$ such that 
\begin{align*}
\Big|\langle g_1\circ L^n, g_2\rangle -\big(\int_{\TT^d}g_1(x)dx\big) \big(\int_{\TT^d}g_2(x)dx\big)\Big| \,\leq\,
K\norm{g_1}_{C^\a}\norm{g_2}_{C^\a}\,e^{-\gamma|n|} \,\text{ for all $n\in\Z$}.
\end{align*}
\end{theorem}

Now we estimate the last line of \eqref{for:19}.
The norms $\|\rho_i^{-n}L_i^{n}\|$ and $\|B_n\|$  grow at most  as $K_1n^d$  since the moduli of all their eigenvalues are at most one. Also, the derivative $D_{E^{i,\ell}}(h_i\circ H^{-1})$ is $\a$-H\"older with zero average, and $\partial_{x_j}(H\circ \Gamma^{-1})_{\Gamma_i\circ H^{-1}y}$ and $J(y)$ 
are $\a$-H\"older functions as well.
Hence  we can use Theorem \ref{th:1} to estimate
\begin{align}\label{for:20a}
&\|\rho_i^{-n}L_i^{n}\big\langle D_{E^{i,\ell}}(h_i\circ H^{-1})_{L^{-n}y}\cdot B_n\cdot \partial_{x_j}(H\circ \Gamma_i^{-1})_{\Gamma_i\circ H^{-1}y},\,
 D^m_{E^u}(\psi)_{\Gamma_i\circ H^{-1}y} \cdot J(y) \big\rangle \|\notag\\
 &\leq K_2\,n^{2d}\, \norm{D_{E^{i,\ell}}(h_i\circ H^{-1})}_{C^\a}\, 
 \norm{\partial_{x_j}(H\circ \Gamma_i^{-1})_{\Gamma_i\circ H^{-1}y} \cdot D^m_{E^u}(\psi)\cdot J(y)}_{C^\a}\cdot e^{-\gamma n}.
\end{align}
Thus  the pairing \eqref{for:19} decays exponentially, specifically, there is $C=C(K_2,h_i,H,\Gamma_i,J)$ such that
\begin{align}\label{for:20}
 \|L_i^{n}\,\big\langle h_i\circ f^{-n}\circ \Gamma_i^{-1}, \, D^m_{E^u}\partial_{x_j}\psi \big\rangle \|
 \leq Cn^{2d} e^{-\gamma n}\,\norm{D^m_{E^u}(\psi)}_{C^\a} \; \text{ for all $n\in \N$}.
\end{align}


Now we estimate the terms in representation \eqref{hi series}, which are similar to the error term 
\eqref{for:19} with
$h_i$ replaced by $R_i$. However, we want to estimate by H\"older norm of $\psi$ in place of 
$\norm{D^m_{E^u}(\psi)}$ and so we want to move the derivative $D^m_{E^u}$ to the left and use
differentiability of $R=f-L$:
\begin{align}\label{for:18}
 \big\langle R_i\circ f^{-k}\circ \Gamma_i^{-1}, \, D^m_{E^u}\partial_{x_j}\psi \big\rangle
 =(-1)^{|m|+1}\big\langle D^m_{E^u}\partial_{x_j}(R_i\circ f^{-k}\circ \Gamma_i^{-1}), \, \psi \big\rangle.
\end{align}
However, having higher order derivatives on the left does not allow to use exponential mixing directly.

Instead,  we split $\psi$ into its truncation smoothing $\psi_N$ as in \eqref{trunc} and the error $\psi-\psi_N$.
Then for any $\b>0$ and $N\ge 1$ the following estimates hold
\begin{align} 
 \norm{\psi-\psi_N}_{L^2}&\leq \,N^{-\b}\norm{\psi}_{\mathcal{H}^\b} 
 \quad\text{and }\label{for:107}\\
 \norm{\psi_N}_{\mathcal{H}^\b}&\leq 2^{\b/2}\,N^{\b}\norm{\psi}_{L^2}. \label{for:142}
\end{align}
Indeed,
$$
 \norm{\psi-\psi_N}_{L^2}^2 =\sum_{\|n\| > N}  |\widehat{\psi}_n|^2 \le 
N^{-2\b} \sum_{\|n\| > N}  \| n\|^{2\b}|\widehat{\psi}_n|^2 \le N^{-2\b} \norm{\psi}_{\mathcal{H}^\b} ^2\quad\text{ and}
$$
$$
 \norm{\psi_N}_{\mathcal{H}^\b}^2 =\sum_{\|n\| \le N} (1+ \| n\|^2)^{\b} |\widehat{\psi}_n|^2 \le 
(1+ N^2)^{\b}\sum_{\|n\| > N} |\widehat{\psi}_n|^2 \le 2^\b N^{2\b} \norm{\psi}_{L^2} ^2.
$$

To estimate the pairing with $\psi_N$ we use bounds  \eqref{for:142} on the derivative 
$D^m_{E^u}\psi_N$ in terms of $\psi$. This allows us to repeat the
estimates \eqref{for:19} and \eqref{for:20} above replacing  $n$ by $k$, $h_i$ by $R_i$, and $\psi$ by $\psi_N$. Thus we obtain
\begin{align*}
 &\Big\|L_i^{k-1}\big\langle R_i\circ f^{-k}\circ \Gamma_i^{-1}, \, D^m_{E^u}\partial_{x_j}(\psi_N) \big\rangle\Big\| \le  Ck^{2d}
 \cdot \norm{D^m_{E^u}(\psi_N)}_{C^\a}\cdot e^{-\gamma k}\\
 &\leq C_1k^{2d} \cdot \norm{\psi_N}_{C^{|m|+1}}\cdot e^{-\gamma k}
 \leq C_1k^{2d}\cdot \norm{\psi_N}_{\mathcal{H}^{|m|+1+d}}\cdot e^{-\gamma k} 
 \quad\text{using  \eqref{for:124}}\\
  &\leq C_2 k^{2d}\cdot N^{|m|+d+1}\cdot\norm{\psi}_{L^2}\cdot e^{-\gamma k}
 \quad\text{using  \eqref{for:142}}.
\end{align*}


To estimate the pairing with $\psi_N -\psi$ we use the bounds on norms 
higher order derivatives of $f$, which we obtain in Lemma \ref{D^mD} in Section \ref{Dest} below. 
Specifically, using the second part of \eqref{D^mDest} with $g=f^{-1}$, $\phi=R_i$, $\w=\w^u$, $\w'=\w^{i,\ell}$, and $\la=\rho_i^{-1}$ we obtain that for each $m$ and $\delta>0$ there exists 
$K=K(\delta, m, \norm{R_i|_{\w^u}}_{C^{m+1}})$ such that for all $k \in \N$, 
$$ \norm{D^{m}_{\w^u} (D_{\w^{i,\ell}}( R_i\circ f^{-k}))}_{C^0}
\leq K(\rho_i^{-1}+\delta)^k\,   \norm{R_i|_{\w^u}}_{C^{m+1}}\cdot 
 $$
 Using this and \eqref{for:18} we estimate
\begin{align*}
 &\Big\|L_i^{k-1}\big\langle R_i\circ f^{-k}\circ \Gamma_i^{-1}, \, D^m_{E^u}\partial_{x_j}(\psi-\psi_N) \big\rangle\Big\|\\
 & = \Big\|L_i^{k-1}\big\langle D^m_{E^u}\partial_{x_j} (R_i\circ f^{-k}\circ \Gamma_i^{-1}), \, (\psi-\psi_N) \big\rangle\Big\|\\
 &\leq \norm{L_i^{k-1}}\cdot  \norm{D^{m}_{\w^u} (D_{\w^{i,\ell}}( R_i\circ f^{-k}))}_{C^0}
 \cdot \norm{\psi-\psi_N}_{L^2}\\
&\leq Ck^d\rho_i^k\cdot  \norm{R_i|_{\w^u}}_{C^{m+1}}\cdot 
 K(\rho_i^{-1}+\delta)^k   \cdot \norm{\psi-\psi_N}_{L^2}\\
 &\leq C_3 k^{d}\cdot (1+\rho_i \delta)^k\cdot N^{-\b}\norm{\psi}_{\mathcal{H}^\b} \quad\text{using  \eqref{for:107}}.
\end{align*}
Finally we combine the two estimates above to get an estimate for terms in \eqref{hi series} 
$$
\| L_i^{k-1}\big\langle R_i\circ f^{-k}\circ \Gamma_i^{-1}, \, D^m_{E^u}\partial_{x_j}\psi \big\rangle\|
\,\le\,  C_4k^{2d} \,\norm{\psi}_{\mathcal{H}^\b} \,(N^{|m|+d+1} e^{-\gamma k} +(1+\rho_i \delta)^kN^{-\b}).
 $$ 
For each $k$ we choose $N=N(k)=e^{{\gamma k}/{(2(|m|+d+1))}}$ 
so that we can write the last term as 
$$
N^{|m|+d+1} e^{-\gamma k} +(1+\rho_i \delta)^kN^{-\b}=
e^{-\gamma k/2}+(1+\rho_i \delta)^ke^{-{\b \gamma k}/{(2(|m|+d+1))}}
 $$ 
 and get an exponentially converging series for small enough $\delta$.
 
 We conclude that there is $0<\xi<1$ and a constant $C$ such that for all $k\in \N$ we have
\begin{align} \label{hi terms est}
\| L_i^{k-1}\big\langle R_i\circ f^{-k}\circ \Gamma_i^{-1}, \, D^m_{E^u}\partial_{x_j}\psi \big\rangle\|
\le  C\xi^k  \norm{\psi}_{\mathcal{H}^\b} 
  \end{align}
  and hence
\begin{align} \label{hi series est}
\|  \big\langle D^m_{E^u}\partial_{x_j} ( h_{i}\circ \Gamma_i^{-1}), \, \psi \big\rangle \|=\|
   \sum_{k=1}^\infty L_i^{k-1}\big\langle R_i\circ f^{-k}\circ \Gamma_i^{-1}, \, D^m_{E^u}\partial_{x_j}\psi \big\rangle \| \le  C_m  \norm{\psi}_{\mathcal{H}^\b}.
  \end{align}


Now we complete the proof of Proposition \ref{L^2} by showing that 
$D^m_{E^u}\partial_{x_j} (h_{i}\circ \Gamma_i^{-1}) \in L^2(\T^d)$ for any $m$.
We denote $\omega =\partial_{x_j} (h_{i}\circ \Gamma_i^{-1})$ and recall that it is 
$\a$-H\"older since the conjugacy $H$ and the global chart $\Gamma$ are $C^{1+\a}$
diffeomorphisms. It follows that $\omega$ is in $\mathcal {H}^{\b}$ for any $\b<\a$ by \eqref{Holder}.  
Then \eqref{hi series est} shows that for any $m$
and any $C^\infty$ function $\psi$ we have
\begin{align*} 
|  \big\langle D^{2m}_{E^u}\omega, \, \psi \big\rangle |=|  \big\langle \omega, \, D^{2m}_{E^u}\psi \big\rangle | \le  C_{2m}  \norm{\psi}_{\mathcal{H}^\b}.
  \end{align*}
  Hence we can apply Lemma \ref{to L^2} to conclude that $D^m_{E^u}\omega$ is in $L^2(\T^d)$.

\subsection{Derivative estimates} \label{Dest}
In this section we prove the derivative estimates used above. While estimates of derivatives 
of compositions along the stable manifold are not new, a precise reference is hard to give
as we need a specific result involving faster part. We will also need similar estimates along 
the center foliation in the next section. 
For a function $\phi$, a foliation
$\w$, and $m\in \N$ we will denote by $D^{m}_{\w} \phi$ the derivative of order $m$ of $\phi$ 
restricted to the leaves of $\w$. We view it here as an $m$-linear form on $T\w$ and 
we denote its norm by $\| D^{m}_{\w} \phi\|$.


\begin{lemma}\label{D^mf^n} 
Let $\w$ be a foliation of $\T^d$ with uniformly $C^\infty$ leaves invariant under a 
$C^\infty$ diffeomorphism $g$ such that $\| D g|_{\w}\| \le \sigma$. 
\vskip.1cm 
\begin{itemize}
\item[(i)]If $\sigma<1$ then for each $m$ and $\delta>0$ 
there exists $C=C(\delta, m, \| g|_\w\|_{C^m})$ such that
 \begin{align}\label{D^mfest}
\| D^{m}_{\w} g^n\| \le C(\s+\delta)^n \quad \text{and} \quad
\| D^{m}_{\w}  (\phi \circ g^n)\| \le C(\s+\delta)^n \| \phi|_\w\|_{C^{m}}
 \end{align}
 for any $\phi \in C^\infty(\T^d)$, where $D^{m}_{\w}$ is the derivative
 of order $m$ along $\w$.
 \vskip.1cm 
\item[(ii)] If $\sigma>1$ then for each $m$
there exists $C=C(m, \| g|_\w\|_{C^m})$ such that
 \begin{align}\label{D^mfest c}
\| D^{m}_{\w} g^n\| \le C\s^{mn} \; \text{ and } \;
\| D^{m}_{\w} (\phi \circ g^n)\| \le C\s^{m^2n} \| \phi|_\w\|_{C^{m}}
\;\text{ for any $\phi \in C^\infty(\T^d)$.}
 \end{align}
 \end{itemize}

 \end{lemma}

\begin{proof} (i) 
We abbreviate $D_{\w}$ to $D$ in this proof. We will show inductively that  for some $c$ 
 and all $m \le m_0$ we have $\| D^{m} g^n\| \le c^{m-1}(\s+\delta)^n$ for all  $n \in \N$. 
 In the base case $m=1$ by the assumption we have $\| D g^n\| \le \s^n$ for all $n \in \N$.
 Suppose the estimate holds for derivatives of orders up to $m-1$. 
 
 Now we show that $\| D^{m} g^n\| \le c^{m-1}(\s+\delta)^n$ for all  $n \in \N$ by induction on $n$.
The base case $n=0$ is trivial.
 For the inductive step we apply  Fa\`a di Bruno's formula to $D^{m} g^{n+1}= D^{m}( g \circ g^{n})$. 
 We slightly abuse notations by suppressing the base points, as they are not important in the estimate.
  \begin{align}\label{Faa}
D^{m}( g \circ g^{n}) = \sum_{k_1, \dots , k_m} C_{k_1, \dots , k_m} D^{k} g \,
[(D^{1} g^{n})^{ \otimes k_1} \otimes \dots \otimes (D^{m} g^{n})^{ \otimes k_m} ],
\end{align}
where $k =k_1 +\dots +k_m$ and the sum is taken over all $k_1, \dots , k_m$ such that 
$k_1 +2 k_2+ \dots +mk_m=m$. 
We note that $k_m=0$ unless $k_m=1=k$ and we can separate the corresponding term as
   \begin{align*}
D^{m}( g \circ g^{n}) =  Dg \, [D^{m} g^{n}]  +\sum_{k_1, \dots , k_{m-1}} C_{k_1, \dots , k_{m-1}} D^{k} g \,
[(D^{1} g^{n})^{ \otimes k_1} \otimes \dots \otimes (D^{m-1} g^{n})^{ \otimes k_{m-1}} ].
\end{align*}
We need to show $\|D^{m}( g \circ g^{n})\| \le c^{m-1}(\s+\delta)^{n+1}$ 
provided that $\| D^{m} g^n\| \le c^{m-1}(\s+\delta)^n$, which yields
$$
\|Dg \, [D^{m} g^{n}] \| \le \|Dg\| \cdot  \| D^{m} g^n\| \le \s \, c^{m-1}(\s+\delta)^n.
$$
Hence it suffices to estimate the norm of the sum from above by the difference
 \begin{align}\label{gap1}
c^{m-1}(\s+\delta)^{n+1} -  \s \, c^{m-1}(\s+\delta)^n =  \delta \, c^{m-1}(\s+\delta)^n.
  \end{align}
We estimate the norms of the terms using inductive assumptions as
   \begin{align*}
\| D^{k} g \|  \prod_{j=1}^{m-1} \| D^{j} g^{n}\| ^{ k_{j}} \,\le\, 
\| g|_\w\|_{C^k} \prod_{j=1}^{m-1}  [ c^{j-1}(\s+\delta)^n]^{k_j}\,=\,\| g|_\w\|_{C^k} \, c^{m-k}(\s+\delta)^{nk}.
\end{align*}
All terms in the sum have $k>1$ and hence each can be estimated as
   \begin{align*}
\|D^{k} g \, [ (D^{1} g^{n})^{ \otimes k_1} \otimes \dots \otimes (D^{m-1} g^{n})^{ \otimes k_{m-1}}]\|\, \le\, 
\| g|_\w\|_{C^k} \, c^{m-2}(\s+\delta)^{n}.
\end{align*}
for $\s+\delta<1$. The ratio of this to \eqref{gap1} is $\| g|_\w\|_{C^k} \, (c \delta)^{-1}$.
Hence the inductive step will hold if we choose $c> N(m)\, \| g|_\w\|_{C^m}\, \delta^{-1} >1$, where $N(m)$ is the 
 sum of the coefficients $C_{k_1, \dots , k_{m-1}}$. We can choose the same $c$  for all $m \le m_0$.
Thus the first estimate in the lemma holds with $C=c^{m-1}$ for any $m \le m_0$.
\vskip.1cm
To prove the second estimate in \eqref{D^mfest} we apply  Fa\`a di Bruno's formula to 
$D^{m} ( \phi \circ g^{n})$. Each term can be estimates as
  \begin{align*}
\|D^{k}_{\w} g \, [ (D^{1}_{\w} g^{n})^{ \otimes k_1} \otimes \dots \otimes (D^{m-1}_{\w} g^{n})^{ \otimes k_{m}}]\| \,\le\,
\| \phi|_\w\|_{C^k} \,  \prod_{j=1}^{m} \| D^{j}_{\w} g^{n}\| ^{ k_{j}} \le
\\
\| \phi|_\w\|_{C^k} \prod_{j=1}^{m}  [  c^{j-1}(\s+\delta)^n]^{k_j} \le\, \| \phi|_\w\|_{C^m} \, c^{m-k}(\s+\delta)^{nk}
\le\,  \| \phi|_\w\|_{C^m} \, C(\s+\delta)^{n}
\end{align*}
since $\s+\delta<1$ and $c^{m-k}\le c^{m-1} =C$. The estimate for the sum follows by adjusting $C$.
\vskip.2cm

(ii) The proof of the second part is similar so we just indicate the changes. We look for the
inductive estimate of the form $\| D^{m} g^n\| \le c^{m-1}\s^{mn}$, with the base $m=1$ 
given by the assumption. Writing $D^{m}( g \circ g^{n}) =  Dg \, [D^{m} g^{n}]  +\sum \dots$
we need to estimate $\|\sum \dots \|$ from above by the gap similar to \eqref{gap1}:
$$
 c^{m-1}\s^{m(n+1)} -  \s  c^{m-1}\s^{mn} \,=\,   c^{m-1}\s^{mn}(\s^m-\s) \,\ge\, c^{m-1}\s^{mn}(\s^2-\s)
 \quad\text{ for }m \ge 2. 
$$
The terms in  $\|\sum \dots \|$ have $k>1$ and can be estimated as before
   \begin{align*}
\|D^{k} g \, [ (D^{1} g^{n})^{ \otimes k_1} \otimes \dots \otimes (D^{m-1} g^{n})^{ \otimes k_{m-1}}]\| \le 
\| g|_\w\|_{C^k} \, c^{m-k}\s^{nk} \le \| g|_\w\|_{C^k} \, c^{m-2}\s^{mn}
\end{align*}
as $2\le k\le m$ and $\s>1$. Hence we can again choose $c$ large enough to obtain the estimate
$\| D^{m}_{\w} g^n\| \le C\s^{mn}$ with $C=c^{m-1}$. To prove the second inequality in 
\eqref{D^mfest c} we estimate each term $D^{m} ( \phi \circ g^{n})$ similarly to the above
with $\s>1$
 \begin{align*}
\|D^{k}_{\w} g \, [ (D^{1}_{\w} g^{n})^{ \otimes k_1} \otimes \dots \otimes (D^{m-1}_{\w} g^{n})^{ \otimes k_{m}}]\| \,\le\, 
\| \phi|_\w\|_{C^k} \,  \prod_{j=1}^{m} \| D^{j}_{\w} g^{n}\| ^{ k_{j}} \le
\\
\| \phi|_\w\|_{C^k} \prod_{j=1}^{m}  [  c^{j-1}\s^{mn}]^{k_j} \le\, \| \phi|_\w\|_{C^m} \, c^{m-k}\s^{mnk}
\,\le\,  \| \phi|_\w\|_{C^m} \, C\s^{m^2n}.
\end{align*}
\end{proof}


\begin{lemma}\label{D^mD} 
Let $\w$ and $\w'$ be  foliations of $\T^d$  invariant under a 
$C^\infty$ diffeomorphism $g$ with uniformly $C^\infty$ leaves  such that $\w'$ is a $C^\infty$ foliation of each leaf of $\w$.
Suppose that $\| D g|_{\w}\| \le \sigma<1$ and $\| D g|_{\w'}\| \le \la $.
Then for each $m$ and $\delta>0$ there exists $K=K(\delta, m, \| g|_\w\|_{C^{m+1}})$ such that
for all $n \in \N$, 
\begin{equation}\label{D^mDest}
 \begin{aligned}
&\| D^{m}_{\w} (D_{\w'}g^n)\| \le K(\la+\delta)^n \;\; \text{and} \\
&\| D^{m}_{\w}  D_{\w'}(\phi \circ g^n)\| \le K(\la+\delta)^n \,\| \phi|_\w\|_{C^{m+1}}
 \;\;\text{ for any } \phi\in C^\infty(\T^d).
 \end{aligned}
 \end{equation}
 \hskip1.5cm
\end{lemma}

\begin{proof}
We again abbreviate $D_{\w}$ to $D$ in this proof. We also denote $G_x=D g|_{\w'(x)}$
and  $G^n_x=D g^n|_{\w'(x)}$. We will show inductively that  for some $c$ 
 and all $m \le m_0$ we have 
 $$\| D^{m} G^n_x\| \le c^{m}(\la+\delta)^n\quad\text{ for all  $n \in \N$.}
 $$ 
 In the base case $m=0$ we have $\| G^n_x\| \le \la^n$ for all $n \in \N$ by the assumption.
 Suppose the estimate holds for the derivatives of orders up to $m-1$. 

Now we show that $\| D^{m} G^n_x\| \le c^{m}(\la+\delta)^n$ for all  $n \in \N$ by induction on $n$.
The base case $n=1$ holds if $c$ is large.
 For the inductive step we differentiate $G^{n+1}_x=G_{g^nx} \circ G^n_x$ 
 as the product of two linear maps which depend on $x$:
  \begin{align}\label{PR}
D^{m} (G_{g^nx} \cdot G^n_x) =\sum_{k=0}^m {m\choose k} D^{k} (G_{g^nx}) \cdot D^{m-k} ( G^n_x) =\\
=G_{g^nx} \cdot D^{m} ( G^n_x)+\sum_{k=1}^m {m\choose k} D^{k} (G_{g^nx}) \cdot D^{m-k} ( G^n_x).
  \end{align}

We want to show that $\|D^{m} (G_{g^nx} \cdot G^n_x)\| \le c^{m}(\la+\delta)^{n+1}$ 
provided that $\| D^{m} ( G^n_x)\| \le c^{m}(\la+\delta)^n$, which yields
$$
\|G_{g^nx} \cdot D^{m} ( G^n_x) \| \le \|G_{g^nx}\|  \| D^{m} ( G^n_x)\| \le \la \, c^{m}(\la+\delta)^n.
$$
Hence it suffices to estimate the norm of the sum on the right from above by the difference
 \begin{align}\label{gap}
c^{m}(\la+\delta)^{n+1} -  \la \, c^{m}(\la+\delta)^n =  \delta \, c^{m}(\la+\delta)^n
  \end{align}
Now we estimate the norms of the terms in the sum.  By the inductive assumption we get
   \begin{align*}
\| D^{m-k} ( G^n_x) \|  \le c^{m-k}(\la+\delta)^n.
\end{align*}
To estimate $D^{k} ( G_{g^nx})=D^{k} [ (G \circ g^n)(x)]$ we use Lemma \ref{D^mf^n}(i)
with $\phi=G$:
$$\| D^{k} ( G_{g^nx})\| \le C(\s+\delta)^n \| G|_\w\|_{C^m} \le C(\s+\delta)^n \| g|_\w\|_{C^{m+1}}\le C \| g|_\w\|_{C^{m+1}}.$$
for $\s+\delta<1$. Hence terms in the sum with $k\ge1$  can be estimated  as
   \begin{align*}
\|D^{k} (G_{g^nx}) \cdot D^{m-k} ( G^n_x)\| \le Cg|_\w\|_{C^{m+1}} \cdot c^{m-k}(\la+\delta)^n
\le C c^{m-1}(\la+\delta)^n  \| g|_\w\|_{C^{m+1}}.
\end{align*}
The ratio of this to \eqref{gap} is $C \| g|_\w\|_{C^{m+1}}(c \delta)^{-1}$. Hence the inductive step will hold if 
we choose $c> 2^m C \| g|_\w\|_{C^{m+1}}\delta^{-1} $, where $2^m$ is the 
 sum of the binomial coefficients. Thus we can choose the same $c$  for all $m \le m_0$
and  the lemma holds with $K=c^m$ for any $m \le m_0$.

The second estimate in the lemma follows from the first one by applying 
 Fa\`a di Bruno's formula to $\phi \circ g^{n}$ and adjusting $K$
 in the same way as in  Lemma \ref{D^mf^n}.
\end{proof}


\section{Proof of Theorem \ref{PH}} \label{PHproof}
In this section we prove Theorem \ref{PH} by describing the adjustments we need to make 
in the proof of Theorem  \ref{mainA}.  

Since $L$ is partially hyperbolic we have the $L$-invariant partially hyperbolic splitting 
$$\R^d=E^s\oplus E^c \oplus E^u,$$ where $E^c$ is the sum of all generalized eigenspaces 
of $L$ corresponding to eigenvalues of modulus $1$.
Since  $f$ is $C^{1+\a}$ conjugate to $L$, it is also partially hyperbolic and it preserves the corresponding splitting 
$$T\T^d=\E^s\oplus \E^c \oplus \E^u
$$
into $\a$-H\"older  sub-bundles 
$\E^*=DH^{-1}(E^*)$  for  $*=s,c,u$, called stable, center, and unstable respectively.
Denoting the corresponding foliations for $L$ and $f$ by $W^s$, $W^c$,  $W^u$ and
$\w^s$, $\w^c$, and $\w^u$ respectively, we have $H(\w^{*})=W^{*}$ for $\ast = s,c,u$. 

In general, the center foliation of a partially hyperbolic system may fail to be absolutely continuous,
and the regularity of individual leaves may be lower than that of $f$, depending on the 
rate of expansion/contraction in $\E^c$. However, in our case $C^{1+\a}$ regularity of $H$ 
means that foliations $\w^s$, $\w^c$, and $\w^u$ are $C^{1+\a}$, and hence all
 three are absolutely continuous with the conditional measures on the leaves given by the 
 restriction of the volume form to their tangent spaces. In addition, the growth  of $Df^n|_{\E^c}$ 
  as $n\to \pm \infty$ is the same as for $L^n|_{E^c}$, that is, at most polynomial. 
 This implies that $f$ is so called {\em strongly r-bunched} for any $r$ and 
 hence $\w^{c}$ has uniformly $C^\infty$ leaves \cite{PSW}.
Existence of $H$ also implies that $\E^c\oplus \E^s$,  $\E^c\oplus \E^u$, and $\E^s\oplus \E^u$
are tangent to $C^{1+\a}$ foliations $\w^{cs}$, $\w^{cu}$, and $\w^{su}$ respectively. 
It follows that these foliations also have uniformly $C^\infty$ leaves, see for example \cite[Lemma 4.1]{KS06}.

As in the proof of Theorem  \ref{mainA}, we can assume that $H$ is in the homotopy class 
of the identity satisfying $H(0)=0$ and $f$ is in the homotopy class of $L$ satisfying $f(0)=0$. 

Using the splitting $\R^d= E^s\oplus E^c\oplus E^u$ for $L$ we define the projections $h^*$ and $R_*$ for $\ast = s,u,c$, of $h=H-\Id$ and $R=f-L$ respectively. Similarly 
to the hyperbolic case, projecting the second equation in \eqref{for:78} to $E^*$ we obtain 
for $\ast = s,u,c$
\begin{equation} \label{H^c}
h^*= L_*^{-1} (h_*\circ f)+L_*^{-1} R_*, \quad \text{where } L_*=L|_{E^*}.
\end{equation}
In particular, $h^u$  is given by \eqref{Hu},
and similarly $h^s=-\sum_{k=1}^\infty L_s^{k-1} (R_s \circ f^{-k}).$ Moreover, the following analog 
of Lemma \ref{h^u s smooth} holds.

\begin{lemma} \label{h^u sc smooth}
The unstable component $h^u$ is uniformly $C^\infty$ along $\w^{cs}$. 
The stable component $h^s$ is uniformly $C^\infty$ along $\w^{cu}$.
The center component $h^c$ is uniformly $C^\infty$ along $\w^{su}$.
\end{lemma}

  \begin{proof} The proof is the same since $H^u=\Id_u+h^u$ is
 locally constant along the leaves of the foliation $\w^{sc}$, which  are uniformly $C^\infty$. Similarly, 
 $H^c=\Id_c+h^c$ is locally constant along the leaves of $\w^{su}$, which are  uniformly $C^\infty$.  \end{proof}
  
Now we  explain why $h^u$ and $h^s$ are $C^\infty$ on $\T^d$.
For $h^u$ we use Lemma  \ref{h^u sc smooth} in place of Lemma \ref{h^u s smooth} and modify the charts $\Gamma_i$ by including the center component. Then we can apply the proof of 
Theorem \ref{th ind} without change, as the arguments work within the unstable foliation.
Thus we obtain that $h^u$, and similarly $h^s$, are $C^\infty$ on $\T^d$, 
and hence so are $H^u$ and $H^s$. 

It remains to show that $h^c$ is $C^\infty$ on $\T^d$. 
We give a proof by modifying our arguments for $h^u$.
By Lemma \ref{h^u s smooth} we already have that $h^c$ is uniformly $C^\infty$ along 
$\w^{su}$. Using the next proposition, we obtain global smoothness 
of $h^c$ on $\T^d$ from \cite[Theorem 3]{L01}. The required properties of $\w^c$ 
 were givend above.

\begin{proposition}\label{L^2c}  
$D^m_{\w^c}   h^{c} \in L^2(\T^d)$ for every every multi-index  $m$.
   \end{proposition}
   \begin{proof}
The proof is a significantly simplified version of the proof of Proposition \ref{L^2}.  
Since we study the derivatives of $h^c$ along $\w^c$, so we do not need any further splitting
of $E^c$ and we do not need to separate derivative $\partial_{x_k}$ as in that proposition.
In particular we do not need Proposition \ref{Dioph} to remove $\partial_{x_k}$.

We will use the foliation chart $\Gamma _c$ obtained as in Proposition \ref{chart} by 
smoothing $h^c$ component of $H$ similarly to \eqref{chart h} and  \eqref{chart G}
\begin{align*} 
\Gamma_c=\Id+\tilde h_\varepsilon , \quad \text{ where } \; 
\tilde h_\varepsilon =\big(h^s, {s}_\varepsilon(h^c),  h^u\big).
\end{align*}
The diffeomorphism $\Gamma _c$ is $C^{1+\a}$ on $\T^d$, uniformly $C^\infty$ along $\w^c$, and  satisfies $\Gamma _c(\w^c)=W^c$.

The reason why the neutral case is simpler is that the equation \eqref{H^c} for $h^c$ is already ``neutral" since $L^n|_{E^c}$ has at most polynomial growth. Since this growth is
 dominated by the exponential mixing one can easily see that $h^c$ itself, unlike $h^u$, can be
 written as a series  
 $
h^c =-\sum_{k=1}^\infty L_c^{k-1} (R_c \circ f^{-k}) 
$ 
 in distributional sense. More formally, we obtain the following series representation 
 for $D^m_{\w^c}$ in the same way as \eqref{hi series}  replacing $\rho_i$ with $\rho_c=1$
\begin{align} \label{hc series}
   \big\langle h^{c}\circ \Gamma_c^{-1}, \, D^m_{E^c}\psi \big\rangle=
   \sum_{k=1}^\infty L_c^{k-1}\big\langle R_c\circ f^{-k}\circ \Gamma_c^{-1}, \, D^m_{E^c}\psi\big\rangle.
  \end{align}
Then we use the same estimates as in the proof of \eqref{hi terms est} to obtain
the analog of \eqref{hi series est}
\begin{align} \label{hc series est}
\|  \big\langle D^m_{E^c} ( h^{c}\circ \Gamma_c^{-1}), \, \psi \big\rangle \|=\|
   \sum_{k=1}^\infty L_c^{k-1}\big\langle R_c\circ f^{-k}\circ \Gamma_c^{-1}, \, D^m_{E^c}\psi \big\rangle \| \le  C_m  \norm{\psi}_{\mathcal{H}^\b}.
  \end{align}
The only differences are the absence of $\partial_{x_k}$ term, replacing $\rho_i$ with $\rho_c=1$,
and estimating  norms of higher order derivatives of $f$ using  Lemma \ref{D^mf^n}(ii).
Specifically, we use the second part of \eqref{D^mfest c} with 
$g=f^{-1}$, $\phi=R_c$, $\w=\w^c$, and $\s=1+\delta$ with $\delta$ small enough.

Finally we apply Lemma \ref{to L^2} in the same way as in the proof of Proposition \ref{L^2c} to conclude that $ D^m_{E^c}( h^{c}\circ \Gamma_c^{-1})$ is in $L^2(\T^d)$ for all $m$.
 
 This completes the proof of Theorem \ref{PH}.
  \end{proof}


\end{document}